\documentclass[11pt]{article}

\usepackage{amssymb,amsmath,amsfonts,amsthm}
\usepackage{latexsym}
\usepackage{graphics}
\usepackage{tikz}
\usetikzlibrary{shapes,arrows}
\usepackage{indentfirst}
\usepackage{hyperref}
\usepackage{mathtools}
\usepackage{comment}
\usepackage{bm}
\usepackage{subcaption}
\usepackage{caption}
\usepackage{graphicx} 
\usetikzlibrary{positioning} 
\hypersetup{colorlinks = true, linkcolor = blue, citecolor = blue, urlcolor = blue}
\usepackage{float}

\allowdisplaybreaks

\newtheorem{innercustomthm}{Theorem}
\newenvironment{customthm}[1]
  {\renewcommand\theinnercustomthm{#1}\innercustomthm}
  {\endinnercustomthm}

\newtheorem*{thm*}{Theorem}
\newtheorem{thm}{Theorem}
\newtheorem{lem}[thm]{Lemma}

\newtheorem{cor}[thm]{Corollary}

\newtheorem{ques}[thm]{Question}

\newcommand{\N}{\mathbb{N}}
\newcommand{\Z}{\mathbb{Z}}

\newcommand{\col}{\mathrm{col}}

\begin{document}

\title{On the DP-chromatic Number of Cartesian Products of Critical Graphs}

\author{ Hemanshu Kaul\thanks{Department of Applied Mathematics, Illinois Institute of Technology, Chicago, IL 60616. E-mail: {\tt kaul@iit.edu}}, 
Jeffrey A. Mudrock\thanks{ Department of Mathematics and Statistics, University of South Alabama, Mobile, AL 36688. E-mail: {\tt mudrock@southalabama.edu}}, 
Gunjan Sharma\thanks{Department of Math and Computer Science, Lake Forest College, Lake Forest, IL 60045. E-mail: {\tt gsharma@lakeforest.edu}}
}

\maketitle

\begin{abstract}
DP-coloring (also called correspondence coloring) is a well-studied generalization of list coloring introduced by  Dvo\v{r}\'{a}k and Postle in 2015. The following sharp bound on the DP-chromatic number of the Cartesian product of graphs $G$ and $H$ is known: $\chi_{DP}(G \square H) \leq \text{min}\{\chi_{DP}(G) + \text{col}(H), \chi_{DP}(H) + \text{col}(G) \} - 1$ where $\chi_{DP}(G)$ is the DP-chromatic number of $G$ and $\text{col}(H)$ is the coloring number of $H$.  We seek to understand when $\chi_{DP}(G \square K_{l,t})$ is far from its chromatic number: $\chi(G \square K_{l,t}) = \max \{\chi(G), 2 \}$ in the case that $G$ is a $k$-critical graph with $\chi_{DP}(G)=k$.  In particular, we have $\chi_{DP}(G \square K_{l,t}) \leq k + l$, and for fixed $l$ we wish to find the smallest $t$ for which this upper bound is achieved. This can be viewed as an extension of the classic result that the list chromatic number of $K_{l,t}$ is $l+1$ if and only if $t \geq l^l$. Our results illustrate that the DP color function of $G$, the DP analogue of the chromatic polynomial, provides a concept and tool that is useful for making progress on this problem.

\medskip



\end{abstract}

\section{Introduction}\label{intro}

 \subsection{Graph coloring, list coloring, and DP-coloring}

 In this paper all graphs are nonempty, finite and simple unless otherwise noted.  Generally speaking we follow West~\cite{W01} for terminology and notation. We use $\mathbb{N}$ to denote the set of all natural numbers. For $k \in \mathbb{N}$, $[k]$ denotes the set $\{1,...,k\}$, $\Z_k$ denotes the integers mod $k$, and $[0]$ denotes the empty set. For a graph $G$, $V(G)$ and $E(G)$ denote the vertex set and edge set of $G$ respectively. If $S \subseteq V(G)$, $G[S]$ denotes the subgraph of $G$ induced by $S$.  For any $S_{1},S_{2} \subseteq V(G)$, $E_{G}(S_{1},S_{2})$ denotes the set consisting of the edges in $E(G)$ that have one endpoint in $S_{1}$ and the other in $S_{2}$. The neighborhood of a vertex $v$ in $G$ is denoted by $N_{G}(v)$ or $N(v)$ when the graph is clear from context. The neighborhood of a set of vertices $S\subseteq V(G)$ is defined as $N(S) = \bigcup_{v \in S}N_{G}(v)$. The degree of a vertex $v\in V(G)$ is defined as the number of vertices in $G$ adjacent to $v$; we denote this by $d_G(v)$ or simply $d(v)$ when the graph is clear from context. We let $\Delta(G)$ and $\delta(G)$ denote the maximum degree and minimum degree of the vertices in $V(G)$ respectively. If $G$ and $H$ are vertex disjoint graphs, we write $G+H$ for the disjoint union of $G$ and $H$, and we write $G \vee H$ for the join of $G$ and $H$.

 A \emph{proper $k$-coloring} of a graph $G$ is a function $f$ that assigns an element of $[k]$ to each $v \in V(G)$ such that $f(v) \neq f(u)$ whenever $uv \in E(G)$. We say that $G$ is \emph{$k$-colorable} if it has a proper $k$-coloring. The \emph{chromatic number} of $G$, denoted $\chi(G)$, is the smallest $k \in \mathbb{N}$ such that there exists a proper $k$-coloring of $G$. The \emph{coloring number} of a graph $G$, denoted $\col(G)$, is the smallest integer $d$ for which there exists an ordering, $v_1, \ldots, v_n$, of the vertices of $G$ such that each vertex $v_i$ has at most $d-1$ neighbors among $v_1, \ldots, v_{i-1}$. For example, $\col(K_{l,t}) = \min\{l,t\}+1$.

 For $k \in \N$, the \emph{chromatic polynomial} of a graph $G$, $P(G,k)$, is equal to the number of proper $k$-colorings of $G$. It can be shown that $P(G,k)$ is a polynomial in $k$ of degree $|V(G)|$ (see~\cite{B12}).
 
 List coloring is a generalization of classical vertex coloring.  It was introduced in the 1970s independently by Vizing~\cite{V76} and Erd\H{o}s, Rubin, and Taylor~\cite{ET79}. A \emph{list assignment} of $G$ is a function $L$ on $V(G)$ that assigns a set of colors to each $v \in V(G)$. If $|L(v)| = k$ for each $v \in V(G)$, then $L$ is called a \emph{$k$-assignment} of $G$. The graph $G$ is \emph{$L$-colorable} if there exists a proper coloring $f$ of $G$ such that $f(v) \in L(v)$ for each $v \in V(G)$ (we refer to $f$ as a \emph{proper $L$-coloring} of $G$). We say $L$ is a \emph{bad assignment} of $G$ if $G$ not $L$-colorable. The \emph{list-chromatic number} of $G$, denoted $\chi_{\ell}(G)$, is the smallest $k$ such that there exists a proper $L$-coloring of $G$ for every $k$-assignment $L$ of $G$.  It immediately follows that for any graph $G$, $\chi(G) \leq \chi_\ell(G) \leq \col(G)$. The first inequality may be strict since it is known that the gap between $\chi(G)$ and $\chi_{\ell}(G)$ can be arbitrarily large; for example, $\chi_{\ell}(K_{k,t}) = k+1$ if and only if $t \ge k^k$ but all bipartite graphs are $2$-colorable.  On the other hand, when $\chi(G) = \chi_{\ell}(G)$, we say that $G$ is \emph{chromatic-choosable}.

 In 1990 the notion of chromatic polynomial was extended to list coloring as follows~\cite{AS90}. If $L$ is a list assignment for $G$, we use $P(G,L)$ to denote the number of proper $L$-colorings of $G$. The \emph{list color function} $P_\ell(G,k)$ is the minimum value of $P(G,L)$ where the minimum is taken over all possible $k$-assignments $L$ for $G$.  Since a $k$-assignment could assign the same $k$ colors to every vertex in a graph, it is clear that $P_\ell(G,k) \leq P(G,k)$ for each $k \in \N$. Dong and Zhang~\cite{DZ23} (improving upon results in~\cite{D92},~\cite{T09}, and~\cite{WQ17}) recently showed that for a connected graph $G$ with at least two edges, $P_{\ell}(G,k)=P(G,k)$ whenever $k \geq |E(G)|-1$. 
 
DP-coloring (also called correspondence coloring) is a generalization of list coloring that was introduced by Dvo\v{r}\'{a}k and Postle~\cite{DP15} in 2015. Intuitively, DP-coloring is a generalization of list coloring where each vertex in the graph still gets a list of colors but identification of which colors are different can vary from edge to edge. We now give formal definitions. A \emph{DP-cover} of a graph $G$ is a pair $\mathcal{H} = (L,H)$, where:

\begin{itemize}
            \item $H$ is a graph and $L$ is a function assigning to each $v \in V(G)$ a set $L(v) \subseteq V(H)$,

            \item the sets $L(v)$ for $v \in V(G)$ are: disjoint sets, independent sets in $H$, and satisfy $V(H) = \bigcup_{v \in V(G)} L(v)$, 

            \item if $u \neq v$ and $E_H(L(u),L(v)) \neq \emptyset$, then $uv \in E(G)$ and $E_H(L(u), L(v))$ is a matching.
        \end{itemize}
        We stress that the matchings between $L(u)$ and $L(v)$ for $uv \in E(G)$ need not be perfect (and may even be empty). The vertices of $H$ are referred to as \emph{colors}.  In this paper, cover always refers to a DP-cover. Suppose $\mathcal{H} = (L,H)$ is a cover of $G$.  An \emph{$\mathcal{H}$-coloring} of $G$ is an \emph{independent transversal} of $\mathcal{H}$ which is an independent set $I$ in $H$ of size $|V(G)|$ with the property $|I \cap L(u)|=1$ for each $u \in V(G)$. We say $\mathcal{H}$ is \emph{$k$-fold} if $|L(u)|=k$ for each $u \in V(G)$. We say $\mathcal{H}$ is a \emph{bad $k$-fold cover} of $G$ if $\mathcal{H}$ is $k$-fold and $G$ does not have an $\mathcal{H}$-coloring. A $k$-fold cover $\mathcal{H}$ is a \emph{full cover} if for each $uv \in E(G)$, the matching $E_{H}(L(u),L(v))$ is perfect. The \emph{DP-chromatic number} of $G$, denoted $\chi_{DP}(G)$, is the smallest $k \in \N$ such that $G$ has an $\mathcal{H}$-coloring whenever $\mathcal{H}$ is a $k$-fold cover of $G$. 
        
        Suppose $\mathcal{H} = (L,H)$ is a cover of $G$.  Suppose $U \subseteq V(G)$. Let $\mathcal{H}_{U} = (L_{U},H_{U})$ where $L_{U}$ is the restriction of $L$ to $U$ and $H_{U} = H[\bigcup_{u \in U} L(u)]$. Clearly, $\mathcal{H}_{U}$ is a cover of $G[U]$. We call $\mathcal{H}_{U}$ the \emph{subcover of }$\mathcal{H}$\emph{ induced by }$U$.  Suppose $G'$ is a subgraph of $G$. Let $\mathcal{H}_{G'} = (L_{G'},H_{G'})$ where $L_{G'}$ is the restriction of $L$ to $V(G')$ and $H_{G'}=H[\bigcup_{u \in V(G')} L(u)] - \bigcup_{uv \in E(G) - E(G')}E_H(L(u),L(v))$. We call $\mathcal{H}_{G'}$ the \emph{subcover of }$\mathcal{H}$\emph{ corresponding to }$G'$.

Suppose $\mathcal{H} = (L,H)$ is a $k$-fold cover of $G$.  Then, $\mathcal{H}$ is \emph{canonical} if it admits a \emph{canonical labeling} which is a mapping $\lambda \colon V(H) \to [k]$ such that
        \begin{itemize}
            \item for each $v \in V(G)$, the restriction of ${\lambda}$ to $L(v)$ is a bijection from $L(v)$ to $[k]$, and
            \item for all $uv \in E(G)$ and $c \in L(u)$, $c' \in L(v)$, we have $cc' \in E(H)$ if and only if $\lambda(c) = \lambda(c')$. 
        \end{itemize}  
        
        Suppose $k \in \N$, and $\mathcal{H} = (L,H)$ is a $k$-fold cover of $G$ such that $\mathcal{H}$ has a canonical labeling called $\lambda$.  From this point forward, unless otherwise noted, in such a situation we will let $(v, j)$ be the vertex in $L(v)$ that $\lambda$ maps to $j$ for each $v \in V(G)$ and $j \in [k]$.  Notice that if $\mathcal{I}$ is the set of $\mathcal{H}$-colorings of $G$ and $\mathcal{C}$ is the set of proper $k$-colorings of $G$, the function $f: \mathcal{C} \rightarrow \mathcal{I}$ given by $f(c) = \{ (v,c(v)) : v \in V(G) \}$ is a bijection.

 Given a $k$-assignment $L$ for a graph $G$, it is easy to construct a $k$-fold cover $\mathcal{H}$ of $G$ such that $G$ has an $\mathcal{H}$-coloring if and only if $G$ has a proper $L$-coloring (see~\cite{BK17}).  It follows that $\chi_\ell(G) \leq \chi_{DP}(G)$. In general, for any graph $G$, $\chi(G) \leq \chi_\ell(G) \leq \chi_{DP}(G) \leq \col(G)$ and these inequalities may be strict. 

 The notion of chromatic polynomial was extended to DP-coloring in~\cite{KM20}.  Suppose $\mathcal{H} = (L,H)$ is a cover of graph $G$.  Let $P_{DP}(G, \mathcal{H})$ be the number of $\mathcal{H}$-colorings of $G$.  Then, the \emph{DP color function of $G$}, denoted $P_{DP}(G,k)$, is the minimum value of $P_{DP}(G, \mathcal{H})$ where the minimum is taken over all possible $k$-fold covers $\mathcal{H}$ of $G$.  It is easy to see that for any graph $G$ and $k \in \N$, $P_{DP}(G, k) \leq P_\ell(G,k) \leq P(G,k)$.
 
 Interestingly, list color functions and DP color functions can behave very differently with respect to their corresponding chromatic polynomial. Recall that for a connected graph $G$ with at least 2 edges, $P_{\ell}(G,k) = P(G,k)$ whenever $k \geq |E(G)|-1$. On the other hand, Dong and Yang~\cite{DY21} (extending results of~\cite{KM20}) showed that if $G$ contains an edge $e$ such that the length of a shortest cycle containing $e$ in $G$ is even, then there exists $N \in \mathbb{N}$ such that $P_{DP}(G,k) < P(G,k)$ whenever $k \geq N$. See~\cite{ZD22} for some more recent work on this. Also, in~\cite{DKM222} it was shown that there are infinitely many graphs $G$ with $\chi_{DP}(G)=3$ for which $P_{DP}(G,3) = P(G,3)$ but there is an $N \in \N$ such that $P_{DP}(G,k) < P(G,k)$ whenever $k \geq N$.

 \subsection{Criticality}  \label{subsec:criticality}

Criticality is a notion of fundamental importance in extremal graph theory that is used to study a wide variety of graph properties.  For $k\geq 2$, a \emph{$k$-critical graph} is a graph whose chromatic number is $k$ but whose proper subgraphs have chromatic number strictly less than $k$.  We will also refer to a $k$-critical graph as a \emph{critical graph}. We say a graph $G$ is \emph{vertex $k$-critical graph} if $\chi(G) = k$ and for each $v \in V(G)$, $\chi(G - v) < k$. In 1951, Dirac~\cite{D51} initiated the study of critical graphs and since then this notion has been widely investigated.

The notion of criticality with respect to list coloring has been well studied (see~\cite{STV09}).  In 2018, the second and third named author~\cite{KM18} introduced the notion of strong chromatic-choosability. For $k \geq 2$, a graph $G$ is \emph{strongly $k$-chromatic-choosable} if it is $k$-vertex-critical and every bad $(k-1)$-assignment for $G$ is constant. If $G$ is strongly $k$-chromatic-choosable for some $k$, we say that $G$ is \emph{strongly chromatic-choosable}.

Very recently the authors of this paper along with Bernshteyn defined a notion of criticality using DP-coloring.  A graph $G$ is \emph{robustly $k$-critical} for $k \geq 2$ if it is $k$-critical and every bad $(k-1)$-fold cover of $G$ is canonical. If $G$ is robustly $k$-critical for some $k$, we say that $G$ is \emph{robustly critical}.  For example, odd cycles are robustly $3$-critical~\cite{KMG21}, $K_{n}$ is robustly $n$-critical, and for any $n,m \in \N$, the join of the complete graph $K_{n}$ and the odd cycle $C_{2m+1}$ is robustly $(n+3)$-critical~\cite{D22}.  For several years, these were the only known examples of robustly critical graphs. However, it is now known many other examples exist.  

\begin{thm}[\cite{BKMS24}]\label{theo:joins_make_robust}
        If $G$ is a critical graph with $m$ edges, then for all $t \geq 100\,m^3$, the graph $G \vee K_t$ is robustly critical.
    \end{thm}
Note that robustly critical graphs are strongly chromatic-choosable as well as critical. It is not known whether robustly critical graphs are exactly the intersection of these two families of graphs (the graphs in this intersection are called strongly critical~\cite{STV09}). See~\cite{BKMS24} for examples and further discussion of the relations between these various notions of criticality.

\subsection{Coloring Cartesian Products of Graphs}

In this paper we show how DP color functions and criticality can be applied to DP coloring Cartesian products of graphs.  In this subsection we present some known results on coloring Cartesian products. 

The \emph{Cartesian product} of graphs $G$ and $H$, denoted $G \square H$, is the graph with vertex set $V(G) \times V(H)$ and edges created so that $(u,v)$ is adjacent to $(u',v')$ if and only if either $u=u'$ and $vv' \in E(H)$ or $v=v'$ and $uu' \in E(G)$. Every connected graph has a unique factorization under this graph product (\cite{S60}), and this factorization can be found in linear time and space (\cite{IP07}).  Also, it is well-known that $\chi(G \square H) = \max \{\chi(G), \chi(H) \}$.  For the list-chromatic number, Borowiecki, Jendrol, Kr{\'a}l, and Mi{\v s}kuf~\cite{BJ06} showed the following in 2006.

\begin{thm}[\cite{BJ06}] \label{thm: Borow1}
For any graphs $G$ and $H$, $\chi_\ell(G \square H) \leq \min \{\chi_\ell(G) + \col(H), \col(G) + \chi_\ell(H) \} - 1.$
\end{thm}

For any graph $G$, it is easy to see that Theorem~\ref{thm: Borow1} implies $\chi_\ell(G \square K_{l,t}) \leq \chi_\ell(G) + l$. To demonstrate the sharpness of Theorem~\ref{thm: Borow1}, the second and third named authors~\cite{KM18}, improving on a previous result~\cite{BJ06}, showed that  $\chi_\ell(G \square K_{l,t}) = \chi_\ell(G) + l$ whenever $t \geq (P_\ell(G, \chi_\ell(G) + l - 1))^l$.  When the first factor is strongly chromatic-choosable, we have the following characterization.

\begin{thm} [\cite{KM182}] \label{thm: sccexact}
If $M$ is strongly chromatic-choosable and $\chi(M) \geq l + 1$, then $\chi_{\ell}(M\square K_{l,t}) = \chi_{\ell}(M)+l$ if and only if $t \geq (P_\ell(M, \chi_\ell(M) + l - 1))^l.$
\end{thm}

Our main motivation for this paper was to prove a DP-coloring analogue of Theorem~\ref{thm: sccexact} when $M$ is robustly critical.  Theorem~\ref{thm: Borow1} was generalized to DP-coloring in~\cite{KMG21}.

 \begin{thm} [\cite{KMG21}] \label{thm: cartprod}
    For any graphs $G$ and $H$, $\chi_{DP}(G\square H) \leq \min\{\chi_{DP}(G) + \col(H),\chi_{DP}(H) + \col(G)\} - 1$.
\end{thm}

The following result demonstrates the sharpness of Theorem~\ref{thm: cartprod}.

\begin{thm} [\cite{KMG21}] \label{thm: cartprodcompbipartite}
	For any graph $G$, $\chi_{DP}(G\square K_{l,t}) = \chi_{DP}(G) + l$ whenever $t \geq (P_{DP}(G,\chi_{DP}(G)+l-1))^{l}$.
\end{thm}

The following question is now natural.

\begin{ques} [\cite{KMG21}] \label{ques: 3}
Given a graph $G$ and $l \in \N$, let $f_{DP}(G,l)$ be the function satisfying: $\chi_{DP}(G\square K_{l,t}) \allowbreak= \chi_{DP}(G)+l$ if and only if $t \geq f_{DP}(G,l)$. What is $f_{DP}(G,l)$?
\end{ques}

Since $\chi_{DP}(G\square K_{l,0}) = \chi_{DP}(G) < \chi_{DP}(G) + l$, we have $f_{DP}(G,l)\geq 1$. Moreover by Theorem~\ref{thm: cartprodcompbipartite}, $f_{DP}(G,l) \leq (P_{DP}(G,\chi_{DP}(G)+l-1))^{l}$. Hence $f_{DP}(G,l)$ exists for every graph $G$ and $l \in \N$.  We are interested in studying $f_{DP}(G,l)$ for critical graphs $G$.  It is known that $f_{DP}(C_{2m+1},1)= P_{DP}(C_{2m+1},3)/3 = (2^{2m+1}-2)/3$ and  $f_{DP}(C_{2m+2},1) = P_{DP}(C_{2m+2},3) = 2^{2m+2}-1$ (see~\cite{KMG21}). 

\subsection{Outline and Open Questions}

In this subsection we present an outline of the paper.  In Section~\ref{star}, we prove the following.

\begin{thm} \label{thm: star}
For $k \geq 2$, suppose $G$ is a graph with $\chi_{DP}(G) = k$. Then the following statements hold. \\ 
(i) If $G$ is connected and every bad $(k-1)$-fold cover of $G$ is full, then $\chi_{DP}(G \square K_{1,t}) \leq k$ whenever $t < \frac{P_{DP}(G,k)}{k}$.\\
(ii) If $G$ is $k$-critical, then $\chi_{DP}(G \square K_{1,t}) = k+1$ whenever $t \geq \frac{P(G,k)}{k}$.
\end{thm}

The following result is immediate from Theorem~\ref{thm: star} and the fact that critical graphs are connected.

\begin{cor} \label{cor: star1}
For $k \geq 2$, let $G$ be a robustly $k$-critical graph. Then, $\chi_{DP}(G \square K_{1,t}) = k+1$ if $t \geq \frac{P(G,k)}{k}$ and $\chi_{DP}(G \square K_{1,t}) \leq k$ if $t < \frac{P_{DP}(G,k)}{k}$.
\end{cor}

Robustly critical graphs are the only examples of which we are aware where both statements of Theorem~\ref{thm: star} apply.  So the following question is natural. 

\begin{ques} \label{ques:robustchar}
    For $k \geq 2$, is the following statement true? If $G$ is a $k$-critical graph such that $\chi(G) = \chi_{DP}(G)$ and every bad $(k-1)$-fold cover of $G$ is full, then $G$ is robustly $k$-critical.
\end{ques}

The answer to Question~\ref{ques:robustchar} is yes when $k \in \{2,3\}$ (see~\cite{BKMS24}). It remains open for $k \geq 4$. In general, for any robustly $k$-critical graph $G$, Corollary~\ref{cor: star1} tells us ${P_{DP}(G,k)}/{k} \leq f_{DP}(G,1) \leq {P(G,k)}/{k}$. Note that if $P_{DP}(G,k)=P(G,k)$, $f_{DP}(G,1) = P_{DP}(G,k)/k$ and all such $G$ show the sharpness of the bounds on $t$ in Corollary~\ref{cor: star1}. This leads to the following question. 

\begin{ques} \label{ques: atDP}
If $G$ is robustly $k$-critical, does it follow that $P(G,k) = P_{DP}(G,k)$? 
\end{ques}

The answer to Question~\ref{ques: atDP} is yes when $G$ is a complete graph, an odd cycle (\cite{KM20}), or the join of a complete graph and an odd cycle (\cite{CL1}). This question is part of a broader program initiated by the second and third named authors (see \cite{KM20}) on the study of DP-color functions: For which graphs does $P(G,k) = P_{DP}(G,k)$ for large enough $k$? It should be noted that in the context of list coloring, as stated in~\cite{KM18}, the following question is open as well: for $k \in \N$, does $P_{\ell}(G,k) = P(G,k)$ whenever $G$ is strongly chromatic-choosable? 

In Section~\ref{probproof} we extend Theorem~\ref{thm: star}(ii) to the scenario when the second factor in the Cartesian product is any complete bipartite graph.  We prove the following.  

\begin{thm} \label{thm:probthem}
For $k \geq 2$, let $G$ be a $k$-critical graph such that $\chi_{DP}(G) = k$. For $l \geq 2$, \newline let $c_{k,l} = \left \lceil{\frac{l\ln(k+l-1)}{\ln{((k-1)!)}+l\ln(k+l-1)-\ln((k-1)!(k+l-1)^{l}-(k+l-1)!)}} \right \rceil$. Then, $\chi_{DP}(G \square K_{l,t}) \allowbreak = k+l$ whenever $t \geq c_{k,l}\left(\frac{P(G,k+l-1)}{k+l-1}\right)^l$. Consequently, $\chi_{DP}(G \square K_{l,t}) \allowbreak = k+l$ whenever $t \geq \left(\frac{l(k-1)!\ln(k+l-1)}{(k+l-1)!}\right) \allowbreak(P(G,k+l-1))^l$.
\end{thm}

Notice that when $G$ satisfies the hypotheses of Theorem~\ref{thm:probthem} and $P_{DP}(G,k+l-1) = P(G,k+l-1)$, Theorem~\ref{thm:probthem} gives a result that is an improvement over Theorem~\ref{thm: cartprodcompbipartite}.

\section{General Setup and Proof of Theorem~\ref{thm: star}} \label{star}

In this section, we prove Theorem~\ref{thm: star}.  Throughout we will assume $G$ is a graph with $\chi_{DP}(G)=k \geq 2$ and $V(G) = \{v_1, \ldots, v_n \}$.  We start by defining an equivalence relation. Let $\mathcal{C}$ denote the set of all proper $k$-colorings of $G$. We define the equivalence relation $R$ on $\mathcal{C}$ such that if $c,d\in \mathcal{C}$, then $c R d$ if there exists $j\in\mathbb{Z}_k$ such that $(c(v_i) - d(v_i))\mod k = j$ for all $i\in [n]$.  The following lemma is now immediate.

\begin{lem} \label{lem: latequiclass}
Each equivalence class $\mathcal{E}$ of $R$ as defined above is of size $k$. Furthermore, if $c,d\in\mathcal{E}$ satisfy $c\neq d$, then $c(v_i)\neq d(v_i)$ for all $i\in[n]$.
\end{lem}

Note that Lemma~\ref{lem: latequiclass} gives a partition of the set of all proper $k$-colorings of $G$ into sets of size $k$ which immediately gives a corresponding partition of proper $\mathcal{H}$-colorings of $G$ into sets of size $k$ when $\mathcal{H}$ is a canonical $k$-fold cover of $G$. 

We wish to study the DP-chromatic number of the Cartesian product of a critical graph and a complete bipartite graph.  Let $K$ be a copy of the complete bipartite graph $K_{l,t}$ with partite sets, $X = \{x_{j} : j \in [l]\}$ and $Y = \{y_{q} : q \in [t]\}$. Let $M = G \square K$, $M_{X} = M[\{(v_{i},x_{j}) : i \in [n], j \in [l]\}]$, and for each $q \in [t]$, let $M_{y_{q}} = M[\{(v_{i},y_{q}) : i \in [n]\}]$. Let $\mathcal{H} = (L,H)$ be a $(k+l-1)$-fold cover of $M$.

Now, we recall the definition of \emph{volatile coloring} from~\cite{KMG21}. Intuitively, the notion of volatile coloring formalizes the natural obstruction to extending a partial coloring $I$ of $M_X$ to the rest of the graph $M$. As we will see in Lemma~\ref{lem: nohcoloring}, this turns out to be the only obstruction to coloring the graph $M$.

Let $\mathcal{H}_{X} = (L_{X}, H_{X})$ denote the subcover of $\mathcal{H}$ induced by $V(G)\times X$. Similarly, for each $q \in [t]$, let $\mathcal{H}_{y_{q}} = (L_{y_{q}}, H_{y_{q}})$ denote the subcover of $\mathcal{H}$ induced by $V(G)\times \{y_{q}\}$.  Suppose $I$ is an $\mathcal{H}_{X}$-coloring of $M_{X}$. For each $q \in [t]$, let $D_{q} = \{u \in V(H_{y_{q}}) : N_{H}(u)\cap I = \emptyset\}$ and $H'_{y_{q}} = H[D_{q}]$. For each $u \in V(M_{y_{q}})$, let $L'_{y_{q}}(u) = L_{y_{q}}(u)\cap D_{q}$ and $\mathcal{H}'_{y_{q}} = (L'_{y_{q}}, H'_{y_{q}})$.  We say $I$ is \emph{volatile} for $M_{y_{q}}$, if $\mathcal{H}'_{y_{q}}$ is a bad cover of $M_{y_{q}}$.  The following results from~\cite{KMG21} will be useful for us.

\begin{lem} [\cite{KMG21}] \label{lem: nohcoloring}
For each $\mathcal{H}_{X}$-coloring of $M_{X}$, $I_{X}$, there exists $q \in [t]$ such that $I_{X}$ is volatile for $M_{y_{q}}$ if and only if $\mathcal{H}$ is a bad cover of $M$.

\end{lem}
\begin{lem} [\cite{KMG21}] \label{cor: hcoloring}
Let $c$ be the number of $\mathcal{H}_{X}$-colorings of $M_{X}$. Suppose for each $q \in [t]$, the number of volatile $\mathcal{H}_{X}$-colorings for $M_{y_{q}}$ is at most $z$. If $c > zt$ then $M$ admits an $\mathcal{H}$-coloring.
\end{lem}

We now turn our attention to Theorem~\ref{thm: star} and assume $K$ is a copy of $K_{1,t}$ for $t \in \N$. First, we need an upper bound on the number of volatile $\mathcal{H}_{X}$-colorings for each $M_{y_{q}}$. The next lemma is stated using the notation established above.

\begin{lem} \label{lem: criticalupper}
     Let $G$ be a connected graph such that $\chi_{DP}(G) = k \geq 2$ and every bad $(k-1)$-fold cover of $G$ is full. Let $M = G\square K_{1,t}$ and let $\mathcal{H}$ be a $k$-fold cover of $M$. Then for each $q \in [t]$, $M_{y_{q}}$ has at most $k$ volatile $\mathcal{H}_{X}$-colorings.
     \end{lem}

\begin{proof}
    Suppose the vertices of $G$ are ordered as $v_{1},\hdots,v_{n}$ and $X = \{x\}$. We have $\mathcal{H} = (L,H)$ is a $k$-fold cover of $M$, and suppose $L(u,v) = \{(u,v,i) : i \in [k]\}$ for each $(u,v) \in V(M)$. Let $\mathcal{I}$ be the set of all $\mathcal{H}_{X}$-colorings of $M_{X}$.
         
    Suppose for $r \in [t]$, $M_{y_{r}}$ has at least $k+1$ volatile $\mathcal{H}_{X}$-colorings. Let $J_{1}, J_{2}, \ldots, J_{k+1}$ denote $k+1$ of these colorings. By the Pigeonhole Principle, at least two of these colorings are non-disjoint. Let $I_{1}$ and $I_{2}$ denote two such colorings. Additionally since $I_{1}$ and $I_{2}$ are distinct, there exist $j_{0}, j_{1} \in [n]$ such that $L(v_{j_{0}},x) \cap I_{1} = L(v_{j_{0}},x) \cap I_{2}$ and $L(v_{j_{1}},x) \cap I_{1} \neq L(v_{j_{1}},x) \cap I_{2}$. Since $M_{X}$ is connected, there is a $v_{j_{0}}v_{j_{1}}$-path in $M_{X}$. On this path, there are two adjacent vertices $v_{i_{0}}, v_{i_{1}}$ where $i_{0}, i_{1}\in [n]$ such that $L(v_{i_{0}},x) \cap I_{1} = L(v_{i_{0}},x) \cap I_{2} = \{(v_{i_{0}}, x, p_{0})\}$, $L(v_{i_{1}},x) \cap I_{1} = \{(v_{i_{1}}, x, q_{0})\}$, and $L(v_{i_{1}},x) \cap I_{2} = \{(v_{i_{1}}, x, q_{1})\}$ where $q_{1} \neq q_{0}$ (note that $p_{0},q_{0},q_{1} \in [k]$). Without loss of generality, we suppose $i_{0} = 1$ and $i_{1} = 2$.
    
   For each $j \in [2]$, let $D_{j} = \{u \in V(H_{y_{r}}) : N_{H}(u)\cap I_{j} = \emptyset\}$ and $H^{j}_{y_{r}} = H[D_{j}]$. For each $u \in V(M_{y_{r}})$, we let $L^{j}_{y_{r}}(u) = L_{y_{r}}(u)\cap D_{j}$ and $\mathcal{H}^{j}_{y_{r}} = (L^{j}_{y_{r}}, H^{j}_{y_{r}})$. Since $L(v_1,\allowbreak x) \cap I_{1} = L(v_1,\allowbreak x) \cap I_{2} = \{(v_1,\allowbreak x,\allowbreak p_{0})\}$ and $(v_1,\allowbreak x,\allowbreak p_{0})$ has a unique neighbor in $H_{y_{r}}$, we have $L^{1}_{y_{r}}(v_1,\allowbreak y_{r}) = L^{2}_{y_{r}}(v_1,\allowbreak y_{r}) = L(v_1,\allowbreak y_{r}) - N_{H}((v_1,\allowbreak x,\allowbreak p_{0}))$. Let $L(v_1,\allowbreak y_{r}) - N_{H}((v_1,\allowbreak x,\allowbreak p_{0})) = \{(v_1,\allowbreak y_{r},\allowbreak p_{1}),\ldots\allowbreak (v_1,\allowbreak y_{r},\allowbreak p_{k-1})\}$, and note $\{p_{1},\ldots \allowbreak p_{k-1}\} = [k]-\{p_0\}$. On the other hand, since $q_{0} \neq q_{1}$, \allowbreak $(v_2,\allowbreak x,\allowbreak q_{0})$ and $(v_2,\allowbreak x,\allowbreak q_{1})$ have distinct neighbors in $H_{y_{r}}$. Therefore $|L^{1}_{y_{r}}(v_2,\allowbreak y_{r}) \cap L^{2}_{y_{r}}(v_2,\allowbreak y_{r})| = k-2$. Let $(v_2,\allowbreak y_{r},\allowbreak q_{2}) \in L^{1}_{y_{r}}(v_2,\allowbreak y_{r}) - L^{2}_{y_{r}}(v_2,\allowbreak y_{r})$, $(v_2,\allowbreak y_{r},\allowbreak q_{3}) \in L^{2}_{y_{r}}(v_2,\allowbreak y_{r}) - L^{1}_{y_{r}}(v_2,\allowbreak y_{r})$, and $\{(v_2,\allowbreak y_{r},\allowbreak q_{4}), \ldots, (v_2,\allowbreak y_{r},\allowbreak q_{k+1})\} = L^{1}_{y_{r}}(v_2,\allowbreak y_{r}) \cap L^{2}_{y_{r}}(v_2,\allowbreak y_{r})$ where $\{q_{2},\ldots, \allowbreak q_{k+1}\} = [k]$. To summarize, $L^{1}_{y_{r}}(v_1,\allowbreak y_{r}) = L^{2}_{y_{r}}(v_1,\allowbreak y_{r}) = \{(v_1,\allowbreak y_{r},\allowbreak p_{1}),\ldots\allowbreak (v_1,\allowbreak y_{r},\allowbreak p_{k-1})\}$, $L^{1}_{y_{r}}(v_2,\allowbreak y_{r}) = \{(v_2,\allowbreak y_{r},\allowbreak q_{2}),\allowbreak (v_2,\allowbreak y_{r},\allowbreak q_{4}),\ldots, \allowbreak (v_2,\allowbreak y_{r},\allowbreak q_{k+1})\}$ and $L^{2}_{y_{r}}(v_2,\allowbreak y_{r}) = \{(v_2,\allowbreak y_{r},\allowbreak q_{3}),\allowbreak (v_2,\allowbreak y_{r},\allowbreak q_{4}),\ldots,\allowbreak (v_2,\allowbreak y_{r},\allowbreak q_{k+1})\}$ (note that when $k=2$, $L^1_{y_r}(v_2,y_r) = \{(v_2,y_r,q_2)\}$ and $L^2_{y_r}(v_2,y_r) = \{(v_2,y_r,q_3)\}$).

Since $I_{1}$ and $I_{2}$ are volatile for $M_{y_{r}}$, both $\mathcal{H}^{1}_{y_{r}}$ and $\mathcal{H}^{2}_{y_{r}}$ are bad $(k-1)$-fold covers of $M_{y_{r}}$; hence, both these covers must be full. Thus there must exist two vertices in $\{(v_2,\allowbreak y_{r},\allowbreak q_{2}),\ldots,(v_2,\allowbreak y_{r},\allowbreak q_{k+1})\}$ that have a common neighbor in $\{(v_1,\allowbreak y_{r},\allowbreak p_{1}),\ldots,(v_1,\allowbreak y_{r},\allowbreak p_{k-1})\}$ which is a contradiction.
     \end{proof}

Note that the bound on volatile colorings given in Lemma~\ref{lem: criticalupper} is sharp when $G$ is an odd cycle.  We are now ready to complete the proof of Theorem~\ref{thm: star}.

\begin{proof}[Proof of Theorem~\ref{thm: star}]
Throughout the proof, let the vertices of $G$ be ordered as $v_{1},\ldots,v_{n}$ and let $K$ be the complete bipartite graph with bipartition $X = \{x\}$, $Y = \{y_{q} : q \in [t]\}$.  Also, let $M = G\square K$ and $M_{X} = M[\{(v_{i},x) : i \in [n]\}]$. 

We first prove Statement~(i).  Suppose $G$ is connected, every bad $(k-1)$-fold cover of $G$ is full, and $t < {P_{DP}(G,k)}/{k}$ or equivalently $P_{DP}(G,k) > kt$. Let $\mathcal{H} = (L,H)$ be an arbitrary $k$-fold cover of $M$. By Lemma~\ref{lem: criticalupper}, for each $q \in [t]$, there are at most $k$ $\mathcal{H}_X$-colorings of $M_X$ that are volatile for $M_{y_{q}}$.  Since there are at least $P_{DP}(G,k)$ $\mathcal{H}_X$-colorings of $M_X$, Lemma~\ref{cor: hcoloring} tells us $M$ admits an $\mathcal{H}$-coloring. So, $\chi_{DP}(M) \leq k$, and our proof of Statement~(i) is complete.

    Now we turn our attention to Statement~(ii).  Suppose $G$ is $k$-critical and $t\geq {P(G,k)}/{k}$.  We need to show $\chi_{DP}(M) = k+1$. By Theorem~\ref{thm: cartprod}, we have $\chi_{DP}(M) \leq \chi_{DP}(G) + \col(K) - 1 = k+1$. It remains to show that $\chi_{DP}(M) > k$. We will prove this by constructing a bad $k$-fold cover $\mathcal{H} = (L,H)$ of $M$. For each $(u,v) \in V(M)$, we let $L(u,v) = \{(u,v,i) : i \in [k]\}$. Let $V(H) = \cup_{(u,v) \in V(M)} L(u,v)$.  Next, create edges in 
$H$ so that for each $v \in V(K)$, the subcover $\mathcal{H}_{v}$ has a canonical labeling (where $\mathcal{H}_{v}$ denotes the subcover of $\mathcal{H}$ induced by $V(G)\times\{v\}$). Before defining the remaining edges of $H$, we make some observations and establish some notation.

Note that since $\mathcal{H}_{X}$ is a $k$-fold cover of $M_{X}$ with a canonical labeling, there exists a bijection between the set of proper $k$-colorings of $M_{X}$ and the set of $\mathcal{H}_{X}$-colorings of $M_{X}$. Also, $P_{DP}(M_{X},\mathcal{H}_{X}) = P(M_{X},k)$, and we let $d = P(M_{X},k)$. We denote the collection of proper $k$-colorings of $M_{X}$ by $\mathcal{C} = \{c_{i}: i \in [d]\}$ and the collection of all proper $\mathcal{H}_{X}$-colorings of $M_{X}$ by $\mathcal{I}$. Let the bijection $f: \mathcal{C} \rightarrow \mathcal{I}$ be given by $f(c_q) = \{(v_{i},x,c_q(v_{i})) : i\in [n]\}$ where $q \in [d]$. 

By Lemma~\ref{lem: latequiclass}, $\mathcal{C}$ can be partitioned into equivalence classes of size $k$, we denote each of these classes by $\mathcal{E}_{j}$ where $j \in [\frac{d}{k}]$. Clearly the set $\{ f(\mathcal{E}_j) : j\in [\frac{d}{k}]\}$ is a partition of $\mathcal{I}$ into sets of size $k$. We index the elements of $\mathcal{I}$ as $I_1, \ldots, I_d$ so that for each $j \in [\frac{d}{k}]$,  $f(\mathcal{E}_{j}) = \{I_{k(j - 1) + r} : r\in[k]\}$. For each $j\in[\frac{d}{k}]$, we complete the construction of $H$ by creating edges so that $(v_{l},x,f^{-1}(I_{k(j-1) + r})(v_l))(v_{l},y_{j},r) \in E(H)$ for each $l \in [n]$ and $r \in [k]$. By Lemma~\ref{lem: latequiclass}, for each $l \in [n]$, this set of edges is a matching in $H$.

It remains to show that $M$ does not have an $\mathcal{H}$-coloring. Fix $j \in [d/k]$ and $r \in [k]$, and consider $I_{k(j-1)+r}$.  We claim that $I_{k(j-1)+r}$ is volatile for $M[V(G)\times \{y_{j}\}]$. Denote by $H_{j}'$ the subgraph of $H_{y_j}$ induced by the non-neighbors of the vertices in $I_{k(j-1)+r}$.  Also for each $l \in [n]$, let $L'_j(v_l,y_j) = L(v_l,y_j) - \{(v_l,y_j,r)\}$. Recall that $(v_{l},x,f^{-1}(I_{k(j-1) + r})(v_l))(v_{l},y_{j},r) \in E(H)$ for each $l \in [n]$. It follows that $\mathcal{H}' = (L_{j}', H_{j}')$ is a $(k-1)$-fold cover of $M[V(G)\times \{y_{j}\}]$.  Let $\lambda: V(H_j') \rightarrow [k-1]$ be given by $\lambda(v_l,y_j,q) = q$ if $q < r$ and $\lambda(v_l,y_j,q) = q-1$ if $q \ge r$.  Then, $\lambda$ is a canonical labeling of $\mathcal{H}'$. Since $M[V(G)\times \{y_{j}\}]$ is $k$-critical, $\mathcal{H}'$ is a bad cover of $M[V(G)\times \{y_{j}\}]$. Hence each element of $f(\mathcal{E}_{j})$ is volatile for $M[V(G)\times \{y_{j}\}]$. So, each element of $\mathcal{I}$ is volatile for some $M[V(G) \times \{y_{q}\}]$ where $q \in [t]$. By Lemma~\ref{lem: nohcoloring},  $M$ does not have an $\mathcal{H}$-coloring.    
\end{proof}

\section{Proof of Theorem~\ref{thm:probthem}} \label{probproof}

\begin{customthm} {\bf \ref{thm:probthem}}
For $k \geq 2$, let $G$ be a $k$-critical graph such that $\chi_{DP}(G) = k$. For $l \geq 2$, \newline let $c_{k,l} = \left \lceil{\frac{l\ln(k+l-1)}{\ln{((k-1)!)}+l\ln(k+l-1)-\ln((k-1)!(k+l-1)^{l}-(k+l-1)!)}} \right \rceil$. Then, $\chi_{DP}(G \square K_{l,t}) \allowbreak = k+l$ whenever $t \geq c_{k,l}\left(\frac{P(G,k+l-1)}{k+l-1}\right)^l$. Consequently $\chi_{DP}(G \square K_{l,t}) \allowbreak = k+l$ if $t \geq \left(\frac{l(k-1)!\ln(k+l-1)}{(k+l-1)!}\right) \allowbreak(P(G,k+l-1))^l$.
\end{customthm}

\begin{proof}
For simplicity of notation, in this proof we will use $c$ instead of $c_{k,l}$.

Let the vertices of $G$ be ordered as $u_{1},\hdots,u_{n}$, and let $K$ be the complete bipartite graph with bipartition $X = \{x_{j} : j \in [l]\}$ and $Y = \{y_{q} : q \in [t]\}$. Let $M = G\square K$ and $M_{X} = M[\{(u_{i},x_{j}) : i \in [n], j \in [l]\}]$.  By Theorem~\ref{thm: cartprod}, we have $\chi_{DP}(M) \leq \chi_{DP}(G) + \col(K) - 1 = k+l$. It remains to show that $\chi_{DP}(M) > k+l-1$. We do this by showing the existence of a bad $(k+l-1)$-fold cover $\mathcal{H} = (L,H)$ of $M$ using a partially random procedure.

For each $v \in V(K)$, let $\mathcal{H}_{v} = (L_{v}, H_{v})$ be a canonical $(k+l-1)$-fold cover of $M[V(G)\times\{v\}]$.
For each $v \in V(K)$, we let $L(u,v) = L_{v}(u,v)$ for every $u \in V(G)$ and create edges so that $H[\bigcup_{u \in V(G)} L(u,v)] = H_{v}$. For simplicity, for each $\lambda \in [k+l-1]$, we denote each vertex $((u,v),\lambda) \in V(H)$ as $(u,v,\lambda)$. Next we use a random procedure to add matchings (possibly empty) between $L(u_{i},x_{j})$ and $L(u_{i},y_{q})$ for each $i \in [n]$, $j \in [l]$ and $q \in [t]$ to complete the construction of $\mathcal{H}$.

Let $\mathcal{H}_{X} = (L_{X},H_{X})$ denote the cover of $M_{X}$ where $L_{X}(u,x_{j}) = L_{x_{j}}(u,x_{j})$ for every $(u,x_{j}) \in V(G) \times X$ and $H_{X} = \bigcup_{j=1}^{l} H_{x_{j}}$. Note that $H_{x_{1}}, H_{x_{2}}, \ldots, H_{x_{l}}$ are pairwise vertex disjoint. Let $\mathcal{C}$ and $\mathcal{I}$ denote the collection of all proper $(k+l-1)$-colorings of $M_X$ and the collection of all $\mathcal{H}_{X}$-colorings of $M_{X}$ respectively. Note that since $\mathcal{H}_{X}$ is a canonical $(k+l-1)$-fold cover of $M_{X}$, the function $f: \mathcal{C} \rightarrow \mathcal{I}$ where $f(h) = \{(u_{i},x_{j},h(u_{i},x_{j})) : i\in [n], j \in [l]\}$ is a bijection. Also $P_{DP}(M_{X},\mathcal{H}_{X}) = P(M_{X},k+l-1)$. Let $d = P(G,k+l-1)$ so that $P_{DP}(M_{X},\mathcal{H}_{X}) = d^l$.

For each $j \in [l]$, let $\mathcal{C}_{j}$ and $\mathcal{I}_{j}$ denote the collection of proper $(k+l-1)$-colorings of $M[V(G)\times \{x_{j}\}]$ and the collection of $\mathcal{H}_{x_{j}}$-colorings $M[V(G)\times \{x_{j}\}]$ respectively. Note that for each $j \in [l]$, the function $f_{j}: \mathcal{C}_{j} \rightarrow \mathcal{I}_{j}$ where $f_{j}(h) = \{(u_{i},x_{j},h(u_{i},x_{j})) : i\in [n]\}$ is a bijection. For each $j \in [l]$, $\mathcal{C}_{j}$ is partitioned into equivalence classes by the equivalence relation $R$ defined at the beginning of Section~\ref{star}.  Moreover, these equivalence classes are of size $(k+l-1)$ by Lemma~\ref{lem: latequiclass}. We let $b = d/(k+l-1)$ and arbitrarily name these equivalence classes $\mathcal{E}_{j,p}$ where $p \in \left[b\right ]$. So, $\{f_{j}(\mathcal{E}_{j,p}) : p \in [b] \}$ is a partition of $\mathcal{I}_j$.

In Figures 1 through 4, we illustrate some key ideas of the proof in the case when $G = K_{3}$ and $l = 2$.  As the proof progresses, we refer to each figure when appropriate. Figure~\ref{fig:1} is an illustration of the equivalence classes $\mathcal{E}_{j,p}$ for each $p \in [6]$ where $j$ is an element of $[2]$. Figure~\ref{fig:2} is an illustration of $f_j(\mathcal{E}_{j,2})$.

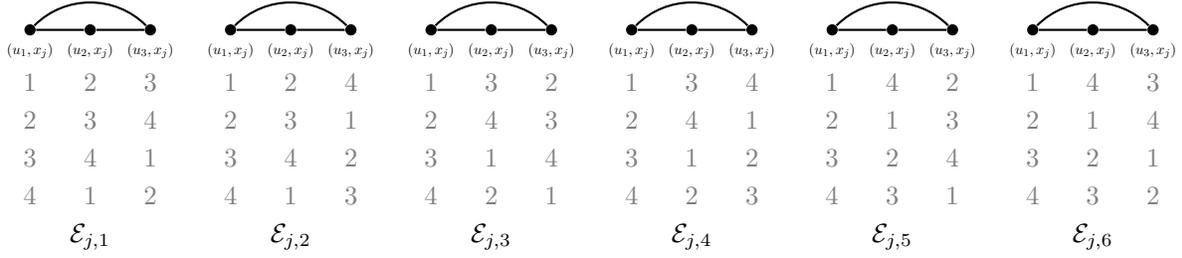
\begin{figure}[H]
\begin{minipage}[b]{0.14\linewidth}
\begin{center}
    \begin{tikzpicture}[scale=0.5]
  \foreach \x/\name in {
  0/{\scalebox{0.5}{\strut$(u_1, x_j)$}},
  1/{\scalebox{0.5}{\strut$(u_2, x_j)$}},
  2/{\scalebox{0.5}{\strut$(u_3, x_j)$}}
} {
  \node[fill=black, circle, inner sep=1.5pt] (v\x) at (\x*1.6,0) {};
  \node[below=2pt of v\x, inner sep=0pt] () {\name};
  
}

\foreach \i/\label in {0/1, 1/2, 2/3} {
  \node[below=0.5cm of v\i, inner sep=0pt, text=gray] {\small \label};
}
\foreach \i/\label in {0/2, 1/3, 2/4} {
  \node[below=1cm of v\i, inner sep=0pt, text=gray] {\small \label};
}
\foreach \i/\label in {0/3, 1/4, 2/1} {
  \node[below=1.5cm of v\i, inner sep=0pt, text=gray] {\small \label};
}
\foreach \i/\label in {0/4, 1/1, 2/2} {
  \node[below=2cm of v\i, inner sep=0pt, text=gray] {\small \label};
}

  \node[below=2.5cm of v1, inner sep=0pt] {\small {$\mathcal{E}_{j,1}$}};

    \draw[thick] (v0) -- (v1);
    \draw[thick] (v1) -- (v2);
    \draw[thick, bend left=-45] (v2) to (v0);
    
    \end{tikzpicture}
   \end{center}
\end{minipage}%
\hfill
\begin{minipage}[b]{0.14\linewidth}
\begin{center}
    \begin{tikzpicture}[scale=0.5]
    \foreach \x/\name in {
  0/{\scalebox{0.5}{\strut$(u_1, x_j)$}},
  1/{\scalebox{0.5}{\strut$(u_2, x_j)$}},
  2/{\scalebox{0.5}{\strut$(u_3, x_j)$}}
} {
  \node[fill=black, circle, inner sep=1.5pt] (v\x) at (\x*1.6,0) {};
  \node[below=2pt of v\x, inner sep=0pt] () {\name};
}

\foreach \i/\label in {0/1, 1/2, 2/4} {
  \node[below=0.5cm of v\i, inner sep=0pt, text=gray] {\small \label};
}
\foreach \i/\label in {0/2, 1/3, 2/1} {
  \node[below=1cm of v\i, inner sep=0pt, text=gray] {\small \label};
}
\foreach \i/\label in {0/3, 1/4, 2/2} {
  \node[below=1.5cm of v\i, inner sep=0pt, text=gray] {\small \label};
}
\foreach \i/\label in {0/4, 1/1, 2/3} {
  \node[below=2cm of v\i, inner sep=0pt, text=gray] {\small \label};
}

  \node[below=2.5cm of v1, inner sep=0pt] {\small {$\mathcal{E}_{j,2}$}};

    \draw[thick] (v0) -- (v1);
    \draw[thick] (v1) -- (v2);
    \draw[thick, bend left=-45] (v2) to (v0);
    \end{tikzpicture}

\end{center}
\end{minipage}%
\hfill
\begin{minipage}[b]{0.14\linewidth}
\begin{center}
    \begin{tikzpicture}[scale=0.5]
    \foreach \x/\name in {
  0/{\scalebox{0.5}{\strut$(u_1, x_j)$}},
  1/{\scalebox{0.5}{\strut$(u_2, x_j)$}},
  2/{\scalebox{0.5}{\strut$(u_3, x_j)$}}
} {
  \node[fill=black, circle, inner sep=1.5pt] (v\x) at (\x*1.6,0) {};
  \node[below=2pt of v\x, inner sep=0pt] () {\name};
}

\foreach \i/\label in {0/1, 1/3, 2/2} {
  \node[below=0.5cm of v\i, inner sep=0pt, text=gray] {\small \label};
}
\foreach \i/\label in {0/2, 1/4, 2/3} {
  \node[below=1cm of v\i, inner sep=0pt, text=gray] {\small \label};
}
\foreach \i/\label in {0/3, 1/1, 2/4} {
  \node[below=1.5cm of v\i, inner sep=0pt, text=gray] {\small \label};
}
\foreach \i/\label in {0/4, 1/2, 2/1} {
  \node[below=2cm of v\i, inner sep=0pt, text=gray] {\small \label};
}

  \node[below=2.5cm of v1, inner sep=0pt] {\small {$\mathcal{E}_{j,3}$}};

    \draw[thick] (v0) -- (v1);
    \draw[thick] (v1) -- (v2);
    \draw[thick, bend left=-45] (v2) to (v0);
    \end{tikzpicture}
    
\end{center}
\end{minipage}%
\hfill
\begin{minipage}[b]{0.14\linewidth}
\begin{center}
    \begin{tikzpicture}[scale=0.5]
    \foreach \x/\name in {
  0/{\scalebox{0.5}{\strut$(u_1, x_j)$}},
  1/{\scalebox{0.5}{\strut$(u_2, x_j)$}},
  2/{\scalebox{0.5}{\strut$(u_3, x_j)$}}
} {
  \node[fill=black, circle, inner sep=1.5pt] (v\x) at (\x*1.6,0) {};
  \node[below=2pt of v\x, inner sep=0pt] () {\name};
}

\foreach \i/\label in {0/1, 1/3, 2/4} {
  \node[below=0.5cm of v\i, inner sep=0pt, text=gray] {\small \label};
}
\foreach \i/\label in {0/2, 1/4, 2/1} {
  \node[below=1cm of v\i, inner sep=0pt, text=gray] {\small \label};
}
\foreach \i/\label in {0/3, 1/1, 2/2} {
  \node[below=1.5cm of v\i, inner sep=0pt, text=gray] {\small \label};
}
\foreach \i/\label in {0/4, 1/2, 2/3} {
  \node[below=2cm of v\i, inner sep=0pt, text=gray] {\small \label};
}

  \node[below=2.5cm of v1, inner sep=0pt] {\small {$\mathcal{E}_{j,4}$}};

    \draw[thick] (v0) -- (v1);
    \draw[thick] (v1) -- (v2);
    \draw[thick, bend left=-45] (v2) to (v0);
    \end{tikzpicture}
    
\end{center}
\end{minipage}%
\hfill
\begin{minipage}[b]{0.14\linewidth}
\begin{center}
    \begin{tikzpicture}[scale=0.5]
    \foreach \x/\name in {
  0/{\scalebox{0.5}{\strut$(u_1, x_j)$}},
  1/{\scalebox{0.5}{\strut$(u_2, x_j)$}},
  2/{\scalebox{0.5}{\strut$(u_3, x_j)$}}
} {
  \node[fill=black, circle, inner sep=1.5pt] (v\x) at (\x*1.6,0) {};
  \node[below=2pt of v\x, inner sep=0pt] () {\name};
}

\foreach \i/\label in {0/1, 1/4, 2/2} {
  \node[below=0.5cm of v\i, inner sep=0pt, text=gray] {\small \label};
}
\foreach \i/\label in {0/2, 1/1, 2/3} {
  \node[below=1cm of v\i, inner sep=0pt, text=gray] {\small \label};
}
\foreach \i/\label in {0/3, 1/2, 2/4} {
  \node[below=1.5cm of v\i, inner sep=0pt, text=gray] {\small \label};
}
\foreach \i/\label in {0/4, 1/3, 2/1} {
  \node[below=2cm of v\i, inner sep=0pt, text=gray] {\small \label};
}

  \node[below=2.5cm of v1, inner sep=0pt] {\small {$\mathcal{E}_{j,5}$}};

    \draw[thick] (v0) -- (v1);
    \draw[thick] (v1) -- (v2);
    \draw[thick, bend left=-45] (v2) to (v0);
    \end{tikzpicture}
    
\end{center}
\end{minipage}%
\hfill
\begin{minipage}[b]{0.14\linewidth}
\begin{center}
    \begin{tikzpicture}[scale=0.5]
    \foreach \x/\name in {
  0/{\scalebox{0.5}{\strut$(u_1, x_j)$}},
  1/{\scalebox{0.5}{\strut$(u_2, x_j)$}},
  2/{\scalebox{0.5}{\strut$(u_3, x_j)$}}
} {
  \node[fill=black, circle, inner sep=1.5pt] (v\x) at (\x*1.6,0) {};
  \node[below=2pt of v\x, inner sep=0pt] () {\name};
}

\foreach \i/\label in {0/1, 1/4, 2/3} {
  \node[below=0.5cm of v\i, inner sep=0pt, text=gray] {\small \label};
}
\foreach \i/\label in {0/2, 1/1, 2/4} {
  \node[below=1cm of v\i, inner sep=0pt, text=gray] {\small \label};
}
\foreach \i/\label in {0/3, 1/2, 2/1} {
  \node[below=1.5cm of v\i, inner sep=0pt, text=gray] {\small \label};
}
\foreach \i/\label in {0/4, 1/3, 2/2} {
  \node[below=2cm of v\i, inner sep=0pt, text=gray] {\small \label};
}

  \node[below=2.5cm of v1, inner sep=0pt] {\small {$\mathcal{E}_{j,6}$}};

    \draw[thick] (v0) -- (v1);
    \draw[thick] (v1) -- (v2);
    \draw[thick, bend left=-45] (v2) to (v0);
    \end{tikzpicture}
   
\end{center}
\end{minipage}

\caption{\small{Suppose $G=K_3$, $l=2$, and $j$ is some element of $[2]$. 
 This is an illustration of the equivalence classes:  $\mathcal{E}_{j,1}, \ldots, \mathcal{E}_{j,6}$. Each row under a copy of $G$ denotes the corresponding vertex colors in a proper $4$-coloring of $G$.}}
 \label{fig:1}
\end{figure}

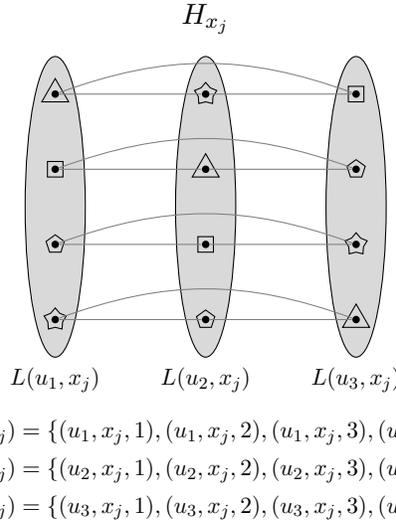
\begin{figure}[H]
\begin{center}
\begin{minipage}[b]{0.45\linewidth}

    \begin{tikzpicture}[scale=1]
    \draw[fill=gray!30] (0,-1.5) ellipse (4mm and 20mm);
    \node[scale=0.8] at (0,-3.8){$L(u_1,x_j)$};
    \node[scale=0.8] at (2,-3.8){$L(u_2,x_j)$};
    \node[scale=0.8] at (4,-3.8){$L(u_3,x_j)$};
    \node[scale=0.8] at (2,-4.5){$L(u_1,x_j) = \{(u_1,x_j,1),(u_1,x_j,2),(u_1,x_j,3),(u_1,x_j,4)\}$};
    \node[scale=0.8] at (2,-5){$L(u_2,x_j) = \{(u_2,x_j,1),(u_2,x_j,2),(u_2,x_j,3),(u_2,x_j,4)\}$};
    \node[scale=0.8] at (2,-5.5){$L(u_3,x_j) = \{(u_3,x_j,1),(u_3,x_j,2),(u_3,x_j,3),(u_3,x_j,4)\}$};
    \draw[fill=gray!30] (2,-1.5) ellipse (4mm and 20mm);

    \node[scale=1] at (2,1){$H_{x_j}$};
    
    \draw[fill=gray!30] (4,-1.5) ellipse (4mm and 20mm);
    
    \foreach \x in {0,1,2} {
        \node[fill=black, circle, inner sep=1pt] (v\x) at (\x*2,0) {}; 
        
    }
\node[regular polygon, regular polygon sides=3, draw, scale=0.02cm] at (v0) {};
\node[regular polygon, regular polygon sides=4, draw, scale=0.02cm] at (v2) {};
\node[star, star points=5, draw, scale=0.02cm] at (v1) {};
    \draw[gray] (v0) -- (v1);
    \draw[gray] (v1) -- (v2);
    \draw[gray, bend left=-20] (v2) to (v0);

    \foreach \y in {0,1,2} {
        \node[fill=black, circle, inner sep=1pt] (u\y) at (\y*2,-1) {}; 
        
    }
\node[regular polygon, regular polygon sides=3, draw, scale=0.02cm] at (u1) {};
\node[regular polygon, regular polygon sides=4, draw, scale=0.02cm] at (u0) {};
\node[regular polygon, regular polygon sides=5, draw, scale=0.02cm] at (u2) {};

    \draw[gray] (u0) -- (u1);
    \draw[gray] (u1) -- (u2);
    \draw[gray, bend left=-20] (u2) to (u0);

     \foreach \z in {0,1,2} {
        \node[fill=black, circle, inner sep=1pt] (w\z) at (\z*2,-2) {}; 
        
    }
\node[regular polygon, regular polygon sides=4, draw, scale=0.02cm] at (w1) {};
\node[regular polygon, regular polygon sides=5, draw, scale=0.02cm] at (w0) {};
\node[star, star points=5, draw, scale=0.02cm] at (w2) {};
    \draw[gray] (w0) -- (w1);
    \draw[gray] (w1) -- (w2);
    \draw[gray, bend left=-20] (w2) to (w0);

     \foreach \y in {0,1,2} {
        \node[fill=black, circle, inner sep=1pt] (t\y) at (\y*2,-3) {}; 
        
    }
    \node[regular polygon, regular polygon sides=3, draw, scale=0.02cm] at (t2) {};
    \node[regular polygon, regular polygon sides=5, draw, scale=0.02cm] at (t1) {};
   \node[star, star points=5, draw, scale=0.02cm] at (t0) {};

    \draw[gray] (t0) -- (t1);
    \draw[gray] (t1) -- (t2);
    \draw[gray, bend left=-20] (t2) to (t0);
    
    \end{tikzpicture}

\end{minipage}
\end{center}

\caption{\small{Assume that the third coordinate of the vertices within each gray oval, ordered from top to bottom, is 1, 2, 3, and 4, respectively. This is an illustration of the elements of $f_j(\mathcal{E}_{j,2})$ in $H_{x_j}$.  Specifically, its  elements are the independent transversals consisting of $3$ vertices surrounded by the same shape. For instance, $f_j$ maps the coloring that assigns $1$ to $u_1$, $2$ to $u_2$ and $4$ to $u_3$ to the independent transversal consisting of the $3$ vertices surrounded by triangles.}}
\label{fig:2}
\end{figure}

Now, arbitrarily name the elements of the set $[b]^l$ as $\boldsymbol{p}_{1}, \ldots, \boldsymbol{p}_{b^l}$.  For each $m \in [b^l]$ let $C_{\boldsymbol{p}_m}$ consist of each $h \in \mathcal{C}$ such that for each $j \in [l]$ the restriction of $h$ to $V(G) \times \{x_j\}$ is an element of $\mathcal{E}_{j,z}$ where $z$ is the $j^{th}$ coordinate of $\boldsymbol{p}_m$.  Note that each element of $C_{\boldsymbol{p}_m}$ can be constructed as follows: for each $i \in [l]$ select a coloring in $\mathcal{E}_{j,z}$ where $z$ is the $i^{th}$ coordinate of $\boldsymbol{p}_m$, then take the union of the $l$ selected colorings.  Since $|\mathcal{E}_{j,p}| = k+l-1$ for each $j \in [l]$ and $p \in [b]$, $|C_{\boldsymbol{p}_m}| = (k+l-1)^l$.  Moreover, $\left \{C_{\boldsymbol{p}_{m}} : m \in \left[b^l\right] \right \}$ is a partition of $\mathcal{C}$. Suppose $S_{\boldsymbol{p}_{m}} = f(C_{\boldsymbol{p}_m})$. Clearly $\left \{S_{\boldsymbol{p}_{m}} : m \in \left[b^l\right] \right \}$ is a partition of $\mathcal{I}$. See Figure~\ref{fig:3} for an illustration of $S_{\boldsymbol{p}_{a}}$ where $\boldsymbol{p}_{a} = (2,2)$.

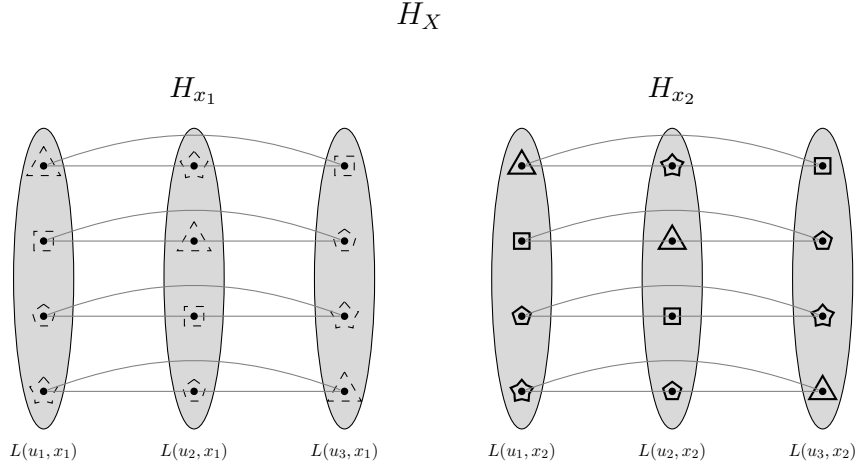
\begin{figure}[ht]
\begin{center}
    
\begin{minipage}[b]{0.2\linewidth}
\begin{center}
    \begin{tikzpicture}[scale=1]
    \draw[fill=gray!30] (0,-1.5) ellipse (4mm and 20mm);
    \draw[fill=gray!30] (2,-1.5) ellipse (4mm and 20mm);
    \node[scale=1] at (2,1){$H_{x_1}$};

   \node[scale=1] at (5,2){\large{$H_{X}$}};


    \draw[fill=gray!30] (4,-1.5) ellipse (4mm and 20mm);
    \node[scale=0.6] at (0,-3.8){$L(u_1,x_1)$};
    \node[scale=0.6] at (2,-3.8){$L(u_2,x_1)$};
    \node[scale=0.6] at (4,-3.8){$L(u_3,x_1)$};

    \foreach \x in {0,1,2} {
        \node[fill=black, circle, inner sep=1pt] (v\x) at (\x*2,0) {}; 
        
    }
\node[regular polygon, regular polygon sides=3, draw,dashed,scale=0.025cm] at (v0) {};
\node[regular polygon, regular polygon sides=4, draw,dashed, scale=0.025cm] at (v2) {};
\node[star, star points=5, draw,dashed, scale=0.025cm] at (v1) {};
    \draw[gray] (v0) -- (v1);
    \draw[gray] (v1) -- (v2);
    \draw[gray, bend left=-20] (v2) to (v0);

    \foreach \y in {0,1,2} {
        \node[fill=black, circle, inner sep=1pt] (u\y) at (\y*2,-1) {}; 
        
    }
\node[regular polygon, regular polygon sides=3, draw,dashed, scale=0.025cm] at (u1) {};
\node[regular polygon, regular polygon sides=4, draw, dashed,scale=0.025cm] at (u0) {};
\node[regular polygon, regular polygon sides=5, draw, dashed,scale=0.025cm] at (u2) {};

    \draw[gray] (u0) -- (u1);
    \draw[gray] (u1) -- (u2);
    \draw[gray, bend left=-20] (u2) to (u0);

     \foreach \z in {0,1,2} {
        \node[fill=black, circle, inner sep=1pt] (w\z) at (\z*2,-2) {}; 
        
    }
\node[regular polygon, regular polygon sides=4, draw, dashed,scale=0.025cm] at (w1) {};
\node[regular polygon, regular polygon sides=5, draw,dashed, scale=0.025cm] at (w0) {};
\node[star, star points=5, draw, dashed,scale=0.025cm] at (w2) {};
    \draw[gray] (w0) -- (w1);
    \draw[gray] (w1) -- (w2);
    \draw[gray, bend left=-20] (w2) to (w0);

     \foreach \y in {0,1,2} {
        \node[fill=black, circle, inner sep=1pt] (t\y) at (\y*2,-3) {}; 
        
    }
    \node[regular polygon, regular polygon sides=3, draw, dashed,scale=0.025cm] at (t2) {};
    \node[regular polygon, regular polygon sides=5, draw, dashed,scale=0.025cm] at (t1) {};
   \node[star, star points=5, draw,dashed, scale=0.025cm] at (t0) {};

    \draw[gray] (t0) -- (t1);
    \draw[gray] (t1) -- (t2);
    \draw[gray, bend left=-20] (t2) to (t0);
    \end{tikzpicture}

\end{center}
\end{minipage}
\hspace{3cm}
\begin{minipage}[b]{0.3\linewidth}
\begin{center}
    \begin{tikzpicture}[scale=1]
    \draw[fill=gray!30] (0,-1.5) ellipse (4mm and 20mm);
    \draw[fill=gray!30] (2,-1.5) ellipse (4mm and 20mm);
    \node[scale=1] at (2,1){$H_{x_2}$};
    
    \draw[fill=gray!30] (4,-1.5) ellipse (4mm and 20mm);
    \node[scale=0.6] at (0,-3.8){$L(u_1,x_2)$};
    \node[scale=0.6] at (2,-3.8){$L(u_2,x_2)$};
    \node[scale=0.6] at (4,-3.8){$L(u_3,x_2)$};
    \foreach \x in {0,1,2} {
        \node[fill=black, circle, inner sep=1pt] (v\x) at (\x*2,0) {}; 
        
    }
\node[regular polygon, regular polygon sides=3, draw,thick, scale=0.02cm] at (v0) {};
\node[regular polygon, regular polygon sides=4, draw, thick,scale=0.02cm] at (v2) {};
\node[star, star points=5, draw, thick,scale=0.02cm] at (v1) {};
    \draw[gray] (v0) -- (v1);
    \draw[gray] (v1) -- (v2);
    \draw[gray, bend left=-20] (v2) to (v0);

    \foreach \y in {0,1,2} {
        \node[fill=black, circle, inner sep=1pt] (u\y) at (\y*2,-1) {}; 
        
    }
\node[regular polygon, regular polygon sides=3, draw,thick, scale=0.02cm] at (u1) {};
\node[regular polygon, regular polygon sides=4, draw, thick,scale=0.02cm] at (u0) {};
\node[regular polygon, regular polygon sides=5, draw,thick, scale=0.02cm] at (u2) {};

    \draw[gray] (u0) -- (u1);
    \draw[gray] (u1) -- (u2);
    \draw[gray, bend left=-20] (u2) to (u0);

     \foreach \z in {0,1,2} {
        \node[fill=black, circle, inner sep=1pt] (w\z) at (\z*2,-2) {}; 
        
    }
\node[regular polygon, regular polygon sides=4, draw,thick, scale=0.02cm] at (w1) {};
\node[regular polygon, regular polygon sides=5, draw,thick, scale=0.02cm] at (w0) {};
\node[star, star points=5, draw, thick,scale=0.02cm] at (w2) {};
    \draw[gray] (w0) -- (w1);
    \draw[gray] (w1) -- (w2);
    \draw[gray, bend left=-20] (w2) to (w0);

     \foreach \y in {0,1,2} {
        \node[fill=black, circle, inner sep=1pt] (t\y) at (\y*2,-3) {}; 
        
    }
    \node[regular polygon, regular polygon sides=3, draw, thick,scale=0.02cm] at (t2) {};
    \node[regular polygon, regular polygon sides=5, draw,thick, scale=0.02cm] at (t1) {};
   \node[star, star points=5, draw,thick, scale=0.02cm] at (t0) {};

    \draw[gray] (t0) -- (t1);
    \draw[gray] (t1) -- (t2);
    \draw[gray, bend left=-20] (t2) to (t0);
    \end{tikzpicture}
   
\end{center}
\end{minipage}
\end{center}

\caption{\small{We choose $a$ so that $\boldsymbol{p}_a$ = (2,2). This is an illustration of $S_{\boldsymbol{p}_{a}}$ in $H_{X}$. Each element of $S_{(2,2)}$ is created by selecting a $3$-element set from $f_1(\mathcal{E}_{1,2})$ and a $3$-element set from $f_2(\mathcal{E}_{2,2})$ and then by taking the union of these two $3$-element sets. Recall from Figure~\ref{fig:2} that each element of $f_1(\mathcal{E}_{1,2})$ (resp. $f_2(\mathcal{E}_{2,2})$) is an independent transversal of $\mathcal{H}_{x_1}$ (resp. $\mathcal{H}_{x_2}$)  and is illustrated as $3$ vertices surrounded by the same dotted (resp. solid) shape. Therefore, each element of $S_{(2,2)}$ is a $6$-element set and is an independent transversal of $\mathcal{H}_{X}$. Furthermore, the cardinality of $S_{(2,2)}$ is $16$. }}
\label{fig:3}
\end{figure}

Note that by Lemma~\ref{lem: nohcoloring}, if $\mathcal{H}$ is such that for each $m \in \left[b^l\right]$, every $s \in S_{\boldsymbol{p}_{m}}$ is volatile for at least one of $M[V(G) \times \{y_{1}\}],\ldots, M[V(G) \times \{y_{t}\}]$, then $M$ does not have an $\mathcal{H}$-coloring. For each $m \in [b^l]$, we associate $c$ ($=c_{k,l}$) copies of $G$, $M[V(G) \times \{y_{c(m-1)+1}\}], \ldots, M[V(G) \times \{y_{c(m-1)+c}\}]$ to $S_{\boldsymbol{p}_{m}}$. Note that this can be done since $t \geq cb^l$.
Next we use a probabilistic argument and show that there exists a way to create matchings between $L(u_{i},x_{j})$ and $L(u_{i},y_{q})$ for each $i \in [n], j \in [l]$, and $q \in \{c(m-1)+1,\ldots, c(m-1)+c\}$ so that each $s \in S_{\boldsymbol{p}_{m}}$ is volatile for at least one of $M[V(G) \times \{y_{c(m-1)+1}\}],\ldots, M[V(G) \times \{y_{c(m-1)+c}\}]$.

Fix $a \in [b^{l}]$, and suppose that $\boldsymbol{p}_{a} = (p_{1}, p_{2}, \ldots, p_{l})$.  For each $j \in [l]$, $f_{j}(\mathcal{E}_{j,p_{j}})$ is a partition of the vertex set of $H_{x_{j}}$. By Lemma~\ref{lem: latequiclass} we may suppose $f_{j}(\mathcal{E}_{j,p_{j}}) = \{I^{j,p_{j}}_{1},\ldots,I^{j,p_{j}}_{k+l-1}\}$. For each $j \in [l]$ and $\omega \in [c]$, we pick a bijection $\sigma_{j,\omega}$ between $f_{j}(\mathcal{E}_{j,p_{j}})$ and $[k+l-1]$ uniformly at random.  For each $i \in [n]$ and $z \in [k+l-1]$, we add an edge between the vertex in $I^{j,p_j}_{z} \cap L(u_{i},x_{j})$ and the vertex $(u_{i},y_{c(a-1)+\omega},\sigma_{j,\omega}(I^{j,p_j}_{z}))$. See Figure~\ref{fig:4} for an illustration of how these edges are added between $H_{x_j}$ and $H_{y_{c(a-1)+\omega}}$.

We repeat the steps from the previous paragraph for each $a \in [b^{l}]$. This completes the (partially random) construction of $\mathcal{H}$. It is easy to verify that $\mathcal{H}$ is a cover. Next we show that for each $m \in [b^l]$, each element of $S_{\boldsymbol{p}_{m}}$ is volatile for at least one of $M[V(G) \times \{y_{c(m-1)+1}\}],\ldots, M[V(G) \times \{y_{c(m-1)+c}\}]$ with positive probability.

Fix $a \in [b^{l}]$. Suppose $\boldsymbol{p}_{a} = (p_{1}, p_{2}, \ldots, p_{l})$.  Also, fix $s \in S_{\boldsymbol{p}_{a}}$.  Let $g = f^{-1}(s)$.  By the definition of $f$ above, $g$ is a proper $(k+l-1)$-coloring of $M_{X}$. For each $j \in [l]$, let $g_j$ be the function $g$ with domain restricted to $V(G) \times \{x_j\}$.  Since $s \in S_{\boldsymbol{p}_{a}}$, $g_j \in \mathcal{E}_{j,p_{j}}$. Then, let $J_{p_j} = f_j(g_j)$. Notice that $J_{p_j}$ is an $\mathcal{H}_{x_{j}}$-coloring of $M[V(G)\times \{x_{j}\}]$. This means $s = \bigcup_{j=1}^{l}J_{p_{j}}$ and $J_{p_{j}} \in f_{j}(\mathcal{E}_{j,p_{j}})$ for each $j \in [l]$. For each $j \in [l]$ and $r \in \{c(a-1)+1,\ldots, c(a-1)+c\}$, suppose the random bijection chosen previously between $f_{j}(\mathcal{E}_{j,p_{j}})$ and $[k+l-1]$ that determines the edges in $H$ between $V(H_{x_j})$ and $V(H_{y_r})$ is $\pi_{j,r}$. For each $r \in \{c(a-1)+1,\ldots, c(a-1)+c\}$, let $D_{r} = \{w \in V(H_{y_{r}}) : N_{H}(w)\cap s = \emptyset\}$ and $H'_{y_{r}} = H[D_{r}]$. Notice $D_r$ consists of exactly the vertices in $V(H_{y_r})$ that don't have a third coordinate in $\{\pi_{j,r}(J_{p_{j}}): j \in [l]\}$. For each $u \in V(G)$, we define $L'_{y_{r}}(u,y_{r}) = L_{y_{r}}(u,y_{r})\cap D_{r}$. Let $\mathcal{H}'_{y_{r}} = (L'_{y_{r}}, H'_{y_{r}})$. Since $\mathcal{H}_{y_{r}}$ is canonical, $\mathcal{H}'_{y_{r}}$ is a canonical $(k+l-1 - \tau)$-fold cover where $\tau=|\{\pi_{j,r}(J_{p_{j}}): j \in [l]\}|$. Since $G$ is $k$-critical, $s$ is volatile for $M[V(G) \times \{y_{r}\}]$ if and only if $\mathcal{H}'_{y_{r}}$ is a canonical $(k-1)$-fold cover of $M[V(G) \times \{y_{r}\}]$ which happens if and only if $\tau = l$. For any $q \in [t]$, Figure~\ref{fig:4} illustrates an example of an $\mathcal{H}_{X}$-coloring that is volatile for $\mathcal{H}_{y_q}$ and an example of an $\mathcal{H}_{X}$-coloring that is not volatile for $\mathcal{H}_{y_q}$.

\begin{figure}[H]

\begin{minipage}[b]{0.2\linewidth}
\begin{center}
\begin{tikzpicture}[scale=0.7]
\draw[fill=gray!30] (0,-1.5) ellipse (4mm and 20mm);
\draw[fill=gray!30] (2,-1.5) ellipse (4mm and 20mm);
\node[scale=1] at (2,1){$H_{x_1}$};

\draw[fill=gray!30] (4,-1.5) ellipse (4mm and 20mm);
\node[scale=0.6] at (0,-3.8){$L(u_1,x_1)$};
    \node[scale=0.6] at (2,-3.8){$L(u_2,x_1)$};
    \node[scale=0.6] at (4,-3.8){$L(u_3,x_1)$};
\foreach \row in {0,...,3} {
    \foreach \col/\shape in {0/3, 1/4, 2/5} {
        \pgfmathsetmacro{\x}{\col*2}
        \pgfmathsetmacro{\y}{-\row}
        \node[fill=black, inner sep=1pt] (a\row\col) at (\x,\y) {};
        
        \ifnum\row=0
            \ifnum\col=0
                \node[regular polygon, regular polygon sides=3, draw, dashed, scale=0.025cm] at (a\row\col) {}; 
            \fi
            \ifnum\col=1
                \node[star, star points=5, draw, dashed, scale=0.025cm] at (a\row\col) {}; 
            \fi
            \ifnum\col=2
                \node[regular polygon, regular polygon sides=4, draw, dashed, scale=0.025cm] at (a\row\col) {}; 
            \fi
        \fi
        
        \ifnum\row=1
            \ifnum\col=0
                \node[regular polygon, regular polygon sides=4, draw, dashed, scale=0.025cm] at (a\row\col) {}; 
            \fi
            \ifnum\col=1
                \node[regular polygon, regular polygon sides=3, draw, dashed, scale=0.025cm] at (a\row\col) {}; 
            \fi
            \ifnum\col=2
                \node[regular polygon, regular polygon sides=5, draw, dashed, scale=0.025cm] at (a\row\col) {}; 
            \fi
        \fi
        
        \ifnum\row=2
            \ifnum\col=0
                \node[regular polygon, regular polygon sides=5, draw, dashed, scale=0.025cm] at (a\row\col) {}; 
            \fi
            \ifnum\col=1
                \node[regular polygon, regular polygon sides=4, draw, dashed, scale=0.025cm] at (a\row\col) {}; 
            \fi
            \ifnum\col=2
                \node[star, star points=5, draw, dashed, scale=0.025cm] at (a\row\col) {}; 
            \fi
        \fi
        
        \ifnum\row=3
            \ifnum\col=0
                \node[star, star points=5, draw, dashed, scale=0.025cm] at (a\row\col) {}; 
            \fi
            \ifnum\col=1
                \node[regular polygon, regular polygon sides=5, draw, dashed, scale=0.025cm] at (a\row\col) {}; 
            \fi
            \ifnum\col=2
                \node[regular polygon, regular polygon sides=3, draw, dashed, scale=0.025cm] at (a\row\col) {}; 
            \fi
        \fi
    }
    \draw[gray] (a\row0) -- (a\row1) -- (a\row2) -- cycle;
    \draw[gray, bend left=20] (a\row0) to (a\row2); 
}
\draw[fill=gray!30] (6,-1.5) ellipse (4mm and 20mm);
\draw[fill=gray!30] (8,-1.5) ellipse (4mm and 20mm);
\node[scale=1] at (8,1){$H_{x_2}$};

\draw[fill=gray!30] (10,-1.5) ellipse (4mm and 20mm);
\node[scale=0.6] at (6,-3.8){$L(u_1,x_2)$};
    \node[scale=0.6] at (8,-3.8){$L(u_2,x_2)$};
    \node[scale=0.6] at (10,-3.8){$L(u_3,x_2)$};
\foreach \row in {0,...,3} {
    \foreach \col/\shape in {0/3, 1/4, 2/5} {
        \pgfmathsetmacro{\x}{\col*2 + 6}
        \pgfmathsetmacro{\y}{-\row}
        \node[fill=black, inner sep=1pt] (b\row\col) at (\x,\y) {};
        
        \ifnum\row=0
            \ifnum\col=0
                \node[regular polygon, regular polygon sides=3, draw, thick, scale=0.02cm] at (b\row\col) {}; 
            \fi
            \ifnum\col=1
                \node[star, star points=5, draw, thick, scale=0.02cm] at (b\row\col) {}; 
            \fi
            \ifnum\col=2
                \node[regular polygon, regular polygon sides=4, draw, thick, scale=0.02cm] at (b\row\col) {}; 
            \fi
        \fi
        
        \ifnum\row=1
            \ifnum\col=0
                \node[regular polygon, regular polygon sides=4, draw, thick, scale=0.02cm] at (b\row\col) {}; 
            \fi
            \ifnum\col=1
                \node[regular polygon, regular polygon sides=3, draw, thick, scale=0.02cm] at (b\row\col) {}; 
            \fi
            \ifnum\col=2
                \node[regular polygon, regular polygon sides=5, draw, thick, scale=0.02cm] at (b\row\col) {}; 
            \fi
        \fi
        
        \ifnum\row=2
            \ifnum\col=0
                \node[regular polygon, regular polygon sides=5, draw, thick, scale=0.02cm] at (b\row\col) {}; 
            \fi
            \ifnum\col=1
                \node[regular polygon, regular polygon sides=4, draw, thick, scale=0.02cm] at (b\row\col) {}; 
            \fi
            \ifnum\col=2
                \node[star, star points=5, draw, thick, scale=0.02cm] at (b\row\col) {}; 
            \fi
        \fi
        
        \ifnum\row=3
            \ifnum\col=0
                \node[star, star points=5, draw, thick, scale=0.02cm] at (b\row\col) {}; 
            \fi
            \ifnum\col=1
                \node[regular polygon, regular polygon sides=5, draw, thick, scale=0.02cm] at (b\row\col) {}; 
            \fi
            \ifnum\col=2
                \node[regular polygon, regular polygon sides=3, draw, thick, scale=0.02cm] at (b\row\col) {}; 
            \fi
        \fi
    }
    \draw[gray] (b\row0) -- (b\row1) -- (b\row2) -- cycle;
    \draw[gray, bend left=20] (b\row0) to (b\row2); 
}

\draw[fill=gray!30] (3,-6.5) ellipse (5mm and 21mm);
\draw[fill=gray!30] (5,-6.5) ellipse (5mm and 21mm);
\node[scale=1] at (5,-9.7){$H_{y_q}$};
\draw[fill=gray!30] (7,-6.5) ellipse (5mm and 21mm);
\node[scale=0.6] at (3,-9){$L(u_1,y_q)$};
    \node[scale=0.6] at (5,-9){$L(u_2,y_q)$};
    \node[scale=0.6] at (7,-9){$L(u_3,y_q)$};
\foreach \row in {0,...,3} {
    \foreach \col in {0,1,2} {
        \pgfmathsetmacro{\x}{\col*2 + 3}
        \pgfmathsetmacro{\y}{-5 - \row}
        \node[fill=black, circle, inner sep=1pt] (c\row\col) at (\x,\y) {};
        \ifnum\row=0
            \node[regular polygon, regular polygon sides=3, draw, thick, scale=0.02cm] at (c\row\col) {};
            \node[regular polygon, regular polygon sides=3, draw, dashed, scale=0.035cm] at (c\row\col) {};
        \fi
    }
\node[regular polygon, regular polygon sides=5, draw, thick, scale=0.02cm] at (3,-6) {};
\node[regular polygon, regular polygon sides=4, draw, dashed, scale=0.04cm] at (3,-6) {};

\node[regular polygon, regular polygon sides=5, draw, thick, scale=0.02cm] at (5,-6) {};
\node[regular polygon, regular polygon sides=4, draw, dashed, scale=0.04cm] at (5,-6) {};

\node[regular polygon, regular polygon sides=5, draw, thick, scale=0.02cm] at (7,-6) {};
\node[regular polygon, regular polygon sides=4, draw, dashed, scale=0.04cm] at (7,-6) {};

\node[regular polygon, regular polygon sides=4, draw, thick, scale=0.02cm] at (3,-7) {};
\node[star, star points=5, draw, dashed, scale=0.04cm] at (3,-7) {};
\node[regular polygon, regular polygon sides=4, draw, thick, scale=0.02cm] at (5,-7) {};
\node[star, star points=5, draw, dashed, scale=0.04cm] at (5,-7) {};
\node[regular polygon, regular polygon sides=4, draw, thick, scale=0.02cm] at (7,-7) {};
\node[star, star points=5, draw, dashed, scale=0.04cm] at (7,-7) {};

\node[regular polygon, regular polygon sides=5, draw, dashed, scale=0.04cm] at (3,-8) {};
\node[star, star points=5, draw, thick, scale=0.02cm] at (3,-8) {};
\node[regular polygon, regular polygon sides=5, draw, dashed, scale=0.04cm] at (5,-8) {};
\node[star, star points=5, draw, thick, scale=0.02cm] at (5,-8) {};
\node[regular polygon, regular polygon sides=5, draw, dashed, scale=0.04cm] at (7,-8) {};
\node[star, star points=5, draw, thick, scale=0.02cm] at (7,-8) {};
    
    \draw[gray] (c\row0) -- (c\row1) -- (c\row2) -- cycle;
    \draw[gray, bend left=20] (c\row0) to (c\row2);
   \draw[gray] (3,-5) -- (0,0);
    \draw[gray] (5,-5) -- (2,-1);
     \draw[gray] (7,-5) -- (4,-3);
      \draw[gray] (3,-5) -- (6,0);
      \draw[gray] (5,-5) -- (8,-1);
      \draw[gray] (7,-5) -- (10,-3);
}

\end{tikzpicture}
\end{center}
\end{minipage}
\hspace{5cm}
\begin{minipage}[b]{0.2\linewidth}
\begin{center}
\vrule width 0.5pt height 2.7\linewidth
\begin{tikzpicture}[scale=0.7]
\draw[fill=gray!30] (0,-1.5) ellipse (4mm and 20mm);
\draw[fill=gray!30] (2,-1.5) ellipse (4mm and 20mm);
\node[scale=1] at (2,1){$H_{x_1}$};

\draw[fill=gray!30] (4,-1.5) ellipse (4mm and 20mm);
\node[scale=0.6] at (0,-3.8){$L(u_1,x_1)$};
    \node[scale=0.6] at (2,-3.8){$L(u_2,x_1)$};
    \node[scale=0.6] at (4,-3.8){$L(u_3,x_1)$};
\foreach \row in {0,...,3} {
    \foreach \col/\shape in {0/3, 1/4, 2/5} {
        \pgfmathsetmacro{\x}{\col*2}
        \pgfmathsetmacro{\y}{-\row}
        \node[fill=black, inner sep=1pt] (a\row\col) at (\x,\y) {};
        
        \ifnum\row=0
            \ifnum\col=0
                \node[regular polygon, regular polygon sides=3, draw, dashed, scale=0.025cm] at (a\row\col) {}; 
            \fi
            \ifnum\col=1
                \node[star, star points=5, draw, dashed, scale=0.025cm] at (a\row\col) {}; 
            \fi
            \ifnum\col=2
                \node[regular polygon, regular polygon sides=4, draw, dashed, scale=0.025cm] at (a\row\col) {}; 
            \fi
        \fi
        
        \ifnum\row=1
            \ifnum\col=0
                \node[regular polygon, regular polygon sides=4, draw, dashed, scale=0.025cm] at (a\row\col) {}; 
            \fi
            \ifnum\col=1
                \node[regular polygon, regular polygon sides=3, draw, dashed, scale=0.025cm] at (a\row\col) {}; 
            \fi
            \ifnum\col=2
                \node[regular polygon, regular polygon sides=5, draw, dashed, scale=0.025cm] at (a\row\col) {}; 
            \fi
        \fi
        
        \ifnum\row=2
            \ifnum\col=0
                \node[regular polygon, regular polygon sides=5, draw, dashed, scale=0.025cm] at (a\row\col) {}; 
            \fi
            \ifnum\col=1
                \node[regular polygon, regular polygon sides=4, draw, dashed, scale=0.025cm] at (a\row\col) {}; 
            \fi
            \ifnum\col=2
                \node[star, star points=5, draw, dashed, scale=0.025cm] at (a\row\col) {}; 
            \fi
        \fi
        
        \ifnum\row=3
            \ifnum\col=0
                \node[star, star points=5, draw, dashed, scale=0.025cm] at (a\row\col) {}; 
            \fi
            \ifnum\col=1
                \node[regular polygon, regular polygon sides=5, draw, dashed, scale=0.025cm] at (a\row\col) {}; 
            \fi
            \ifnum\col=2
                \node[regular polygon, regular polygon sides=3, draw, dashed, scale=0.025cm] at (a\row\col) {}; 
            \fi
        \fi
    }
    \draw[gray] (a\row0) -- (a\row1) -- (a\row2) -- cycle;
    \draw[gray, bend left=20] (a\row0) to (a\row2); 
}
\draw[fill=gray!30] (6,-1.5) ellipse (4mm and 20mm);
\draw[fill=gray!30] (8,-1.5) ellipse (4mm and 20mm);
\node[scale=1] at (8,1){$H_{x_2}$};

\draw[fill=gray!30] (10,-1.5) ellipse (4mm and 20mm);
\node[scale=0.6] at (6,-3.8){$L(u_1,x_2)$};
    \node[scale=0.6] at (8,-3.8){$L(u_2,x_2)$};
    \node[scale=0.6] at (10,-3.8){$L(u_3,x_2)$};
\foreach \row in {0,...,3} {
    \foreach \col/\shape in {0/3, 1/4, 2/5} {
        \pgfmathsetmacro{\x}{\col*2 + 6}
        \pgfmathsetmacro{\y}{-\row}
        \node[fill=black, inner sep=1pt] (b\row\col) at (\x,\y) {};
        
        \ifnum\row=0
            \ifnum\col=0
                \node[regular polygon, regular polygon sides=3, draw, thick, scale=0.02cm] at (b\row\col) {}; 
            \fi
            \ifnum\col=1
                \node[star, star points=5, draw, thick, scale=0.02cm] at (b\row\col) {}; 
            \fi
            \ifnum\col=2
                \node[regular polygon, regular polygon sides=4, draw, thick, scale=0.02cm] at (b\row\col) {}; 
            \fi
        \fi
        
        \ifnum\row=1
            \ifnum\col=0
                \node[regular polygon, regular polygon sides=4, draw, thick, scale=0.02cm] at (b\row\col) {}; 
            \fi
            \ifnum\col=1
                \node[regular polygon, regular polygon sides=3, draw, thick, scale=0.02cm] at (b\row\col) {}; 
            \fi
            \ifnum\col=2
                \node[regular polygon, regular polygon sides=5, draw, thick, scale=0.02cm] at (b\row\col) {}; 
            \fi
        \fi
        
        \ifnum\row=2
            \ifnum\col=0
                \node[regular polygon, regular polygon sides=5, draw, thick, scale=0.02cm] at (b\row\col) {}; 
            \fi
            \ifnum\col=1
                \node[regular polygon, regular polygon sides=4, draw, thick, scale=0.02cm] at (b\row\col) {}; 
            \fi
            \ifnum\col=2
                \node[star, star points=5, draw, thick, scale=0.02cm] at (b\row\col) {}; 
            \fi
        \fi
        
        \ifnum\row=3
            \ifnum\col=0
                \node[star, star points=5, draw, thick, scale=0.02cm] at (b\row\col) {}; 
            \fi
            \ifnum\col=1
                \node[regular polygon, regular polygon sides=5, draw, thick, scale=0.02cm] at (b\row\col) {}; 
            \fi
            \ifnum\col=2
                \node[regular polygon, regular polygon sides=3, draw, thick, scale=0.02cm] at (b\row\col) {}; 
            \fi
        \fi
    }
    \draw[gray] (b\row0) -- (b\row1) -- (b\row2) -- cycle;
    \draw[gray, bend left=20] (b\row0) to (b\row2); 
}

\draw[fill=gray!30] (3,-6.5) ellipse (5mm and 21mm);
\draw[fill=gray!30] (5,-6.5) ellipse (5mm and 21mm);
\node[scale=1] at (5,-9.7){$H_{y_q}$};
\draw[fill=gray!30] (7,-6.5) ellipse (5mm and 21mm);
\node[scale=0.6] at (3,-9){$L(u_1,y_q)$};
    \node[scale=0.6] at (5,-9){$L(u_2,y_q)$};
    \node[scale=0.6] at (7,-9){$L(u_3,y_q)$};
\foreach \row in {0,...,3} {
    \foreach \col in {0,1,2} {
        \pgfmathsetmacro{\x}{\col*2 + 3}
        \pgfmathsetmacro{\y}{-5 - \row}
        \node[fill=black, circle, inner sep=1pt] (c\row\col) at (\x,\y) {};
        \ifnum\row=0
            \node[regular polygon, regular polygon sides=3, draw, thick, scale=0.02cm] at (c\row\col) {};
            \node[regular polygon, regular polygon sides=3, draw, dashed, scale=0.035cm] at (c\row\col) {};
        \fi
    }
\node[regular polygon, regular polygon sides=5, draw, thick, scale=0.02cm] at (3,-6) {};
\node[regular polygon, regular polygon sides=4, draw, dashed, scale=0.04cm] at (3,-6) {};

\node[regular polygon, regular polygon sides=5, draw, thick, scale=0.02cm] at (5,-6) {};
\node[regular polygon, regular polygon sides=4, draw, dashed, scale=0.04cm] at (5,-6) {};

\node[regular polygon, regular polygon sides=5, draw, thick, scale=0.02cm] at (7,-6) {};
\node[regular polygon, regular polygon sides=4, draw, dashed, scale=0.04cm] at (7,-6) {};

\node[regular polygon, regular polygon sides=4, draw, thick, scale=0.02cm] at (3,-7) {};
\node[star, star points=5, draw, dashed, scale=0.04cm] at (3,-7) {};
\node[regular polygon, regular polygon sides=4, draw, thick, scale=0.02cm] at (5,-7) {};
\node[star, star points=5, draw, dashed, scale=0.04cm] at (5,-7) {};
\node[regular polygon, regular polygon sides=4, draw, thick, scale=0.02cm] at (7,-7) {};
\node[star, star points=5, draw, dashed, scale=0.04cm] at (7,-7) {};

\node[regular polygon, regular polygon sides=5, draw, dashed, scale=0.04cm] at (3,-8) {};
\node[star, star points=5, draw, thick, scale=0.02cm] at (3,-8) {};
\node[regular polygon, regular polygon sides=5, draw, dashed, scale=0.04cm] at (5,-8) {};
\node[star, star points=5, draw, thick, scale=0.02cm] at (5,-8) {};
\node[regular polygon, regular polygon sides=5, draw, dashed, scale=0.04cm] at (7,-8) {};
\node[star, star points=5, draw, thick, scale=0.02cm] at (7,-8) {};
    
    \draw[gray] (c\row0) -- (c\row1) -- (c\row2) -- cycle;
    \draw[gray, bend left=20] (c\row0) to (c\row2);
   \draw[gray] (3,-5) -- (0,0);
    \draw[gray] (5,-5) -- (2,-1);
     \draw[gray] (7,-5) -- (4,-3);
      \draw[gray] (3,-7) -- (6,-1);
      \draw[gray] (5,-7) -- (8,-2);
      \draw[gray] (7,-7) -- (10,0);
}

\end{tikzpicture}
\end{center}
\end{minipage}

\caption{\small{This is an illustration of how random matchings between $H_{X}$ and $H_{y_q}$ are created in $H$. Since $k=3$ and $l = 2$, note that $c=2$; we pick $\omega = 1$ and therefore $q = 2a-1$. These matchings are created based on random bijections $\sigma_{1,1}$ and $\sigma_{2,1}$ where $\sigma_{1,1}$ (resp. $\sigma_{2,1}$) is a bijection between $f_{1}(\mathcal{E}_{1,2})$ (resp. $f_{2}(\mathcal{E}_{2,2})$) and $[4]$. Suppose $\sigma_{1,1}$ is such that the images of the elements of $f_1(\mathcal{E}_{1,2})$ illustrated by the $3$ vertices surrounded by the dotted triangles, the dotted squares, the dotted pentagons, and the dotted stars  under $\sigma_{1,1}$ are $1,2,4$ and $3$ respectively. Suppose $\sigma_{2,1}$ is such that the images of the elements of $f_2(\mathcal{E}_{2,2})$ illustrated by the $3$ vertices surrounded by the solid triangles, the solid squares, the solid pentagons, and the solid stars  under $\sigma_{2,1}$ are $1,3,2$ and $4$ respectively. The vertices in $H_{x_1}$ (resp. $H_{x_2}$) surrounded by a particular dotted (resp. solid) shape are adjacent in $H$ to the respective vertices surrounded by the same dotted (resp. solid) shape in $H_{y_q}$. Note that there are a total of $12$ edges between $H_{x_1}$ and $H_{y_q}$, and likewise $12$ edges between $H_{x_2}$ and $H_{y_q}$. For simplicity, only $6$ such edges are shown for each.\\
\textit{Left}: This illustrates that the element of $S_{(2,2)}$ consisting of the $3$ vertices surrounded by dotted triangles in $H_{x_1}$ and the $3$ vertices surrounded by solid triangles in $H_{x_2}$ is not volatile for $H_{y_q}$. There are total $4$ elements of $S_{(2,2)}$ that are not volatile for $H_{y_q}$.\\
\textit{Right}: This illustrates that the element of $S_{(2,2)}$ consisting of $3$ vertices surrounded by dotted triangles in $H_{x_1}$ and $3$ vertices surrounded by solid rectangles in $H_{x_2}$ is volatile for $H_{y_q}$. There are total $12$ elements of $S_{(2,2)}$ that are volatile for $H_{y_q}$.}}
\label{fig:4}

\end{figure}
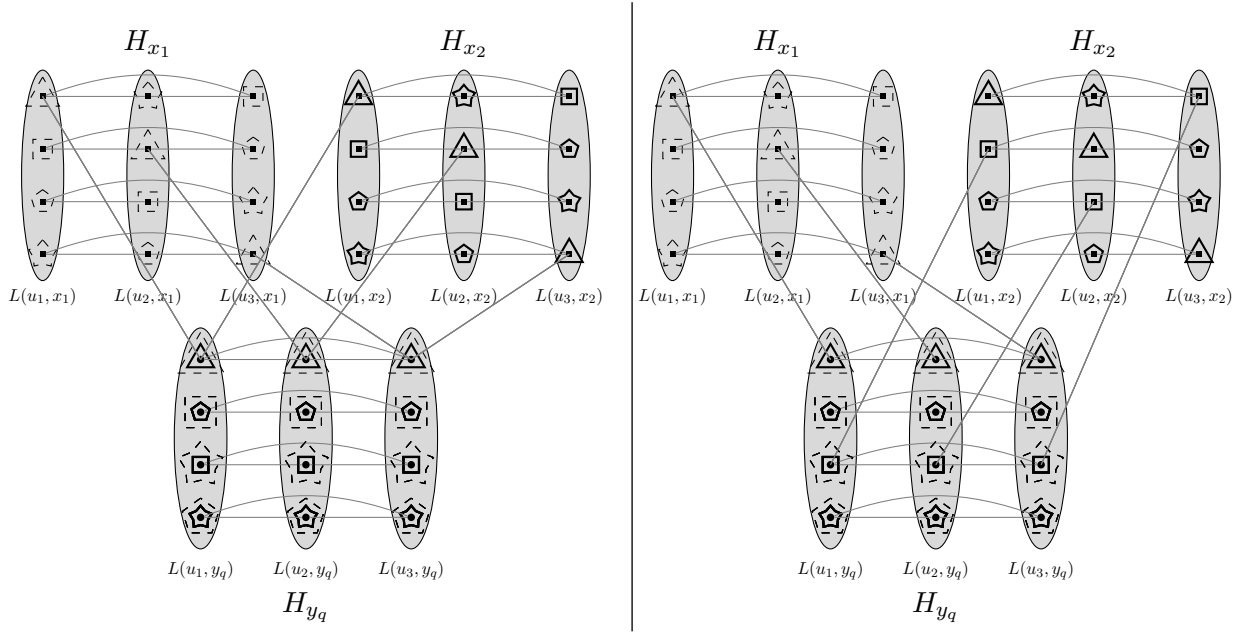

We now calculate the probability that $s$ is volatile for $M[V(G) \times \{y_{r}\}]$. Since the number of possible bijections between $f(\mathcal{E}_{j,p_{j}})$ and $[k+l-1]$ for each $j \in [l]$ is $(k+l-1)!$, there are total $((k+l-1)!)^l$ ways to add edges between $V(H_X)$ and $V(H_{y_r})$ in the way described earlier. We count the number of matchings that correspond to $s$ being volatile for $M[{V(G) \times y_{r}}]$ as follows. For $j=1$, there are $(k+l-1)$ possible values for $\pi_{1,r}(J_{p_{1}})$. Then there are $(k+l-2)!$ possible bijections between $f_1(\mathcal{E}_{1,p_{1}})-\{J_{p_{1}}\}$ and $[k+l-1]-\{\pi_{1,r}(J_{p_{1}})\}$. Now, consider $j=2$. In order for $s$ to be volatile for $M[V(G) \times \{y_{r}\}]$, $\pi_{2,r}(J_{p_{2}})$ must be different from $\pi_{1,r}(J_{p_{1}})$. Thus, there are $k+l-2$ possible values for $\pi_{2,r}(J_{p_{2}})$ and $(k+l-2)!$ possible bijections between $f_2(\mathcal{E}_{2,p_{2}})-\{J_{p_{2}}\}$ and $[k+l-1]-\{\pi_{2,r}(J_{p_{2}})\}$. Continuing in this fashion, once we get to $j=l$, there are $(k+l-1)-(l-1)=k$ possible values for $\pi_{l,r}(J_{p_{l}})$ and $(k+l-2)!$ possible bijections between $f_l(\mathcal{E}_{l,p_{l}})-\{J_{p_{l}}\}$ and $[k+l-1]-\{\pi_{l,r}(J_{p_{l}})\}$. Thus for $l\geq2$, the probability that $s$ is volatile for $M[V(G) \times \{y_{r}\}]$ is 
\begin{equation*}
        \frac{(k+l-1)(k+l-2)!\;(k+l-2)(k+l-2)!\cdots k(k+l-2)!}{((k+l-1)!)^l} = \frac{(k+l-1)!}{(k-1)!(k+l-1)^l}.
    \end{equation*}
 
Suppose $E_{a,s}$ is the event that $s$ is volatile for at least one of $M[V(G) \times \{y_{c(a-1)+1}\}],\allowbreak \ldots,\allowbreak M[V(G) \times \{y_{c(a-1)+c}\}]$. Hence, 

\begin{equation*}
        \mathbb{P}[E_{a,s}] = 1-\left( 1- \frac{(k+l-1)!}{(k-1)!(k+l-1)^l} \right)^c.
    \end{equation*}

Let $X_{a,s}$ be the indicator random variable such that $X_{a,s} = 1$ when $E_{a,s}$ occurs and $X_{a,s} = 0$ otherwise. Let $X_{a} = \Sigma_{s \in S_{\boldsymbol{p}_{a}}} X_{a,s}$. Note that when $X_{a} = (k+l-1)^l$ the event $E_a$ that each element of $S_{\boldsymbol{p}_{a}}$ is volatile for at least one of $M[V(G) \times \{y_{c(a-1)+1}\}],\ldots, M[V(G) \times \{y_{c(a-1)+c}\}]$ occurs. By linearity of expectation we have,
 \begin{equation*}
        \mathbb{E}[X_{a}] = (k+l-1)^l \left [1-\left( 1- \frac{(k+l-1)!}{(k-1)!(k+l-1)^l} \right)^c \right ].
    \end{equation*}
    
Thus if $c$ satisfies
\begin{equation*}
        (k+l-1)^l \left [1-\left( 1- \frac{(k+l-1)!}{(k-1)!(k+l-1)^l} \right)^c \right ] > (k+l-1)^l-1,
\end{equation*} 
then $\mathbb{P}[E_{a}] > 0$.  We can show by a straightforward simplification that when $l \geq 2$ the above inequality holds if and only if 
$$c > \frac{l\ln(k+l-1)}{\ln{((k-1)!)}+l\ln(k+l-1)-\ln((k-1)!(k+l-1)^{l}-(k+l-1)!)}.$$ 

By our construction the events $E_{1}, \ldots, E_{b^l}$ are independent. So, there exists a $(k+l-1)$-fold cover, $\mathcal{H}^{*}$, of $M$ such that for each $m \in [b^l]$, each element of $S_{\boldsymbol{p}_{m}}$ is volatile for at least one of $M[V(G) \times \{y_{c(m-1)+1}\}],\ldots, M[V(G) \times \{y_{c(m-1)+c}\}]$.  Finally, note that by Lemma~\ref{lem: nohcoloring}, $M$ does not admit an $\mathcal{H}^{*}$-coloring and $k+l-1 < \chi_{DP}(M)$.  
\end{proof}

It can be shown that $c_{k,l} \leq \left(\frac{l(k-1)!\ln(k+l-1)}{(k+l-1)!}\right)(k+l-1)^{l}$, hence $\left(\frac{l(k-1)!\ln(k+l-1)}{(k+l-1)!}\right) \allowbreak(P(G,k+l-1))^l$ suffices as a lower bound on $t$, a strong improvement on Theorem~\ref{thm: cartprodcompbipartite} when $G$ satisfies the hypotheses of the theorem.

\end{document}